\documentclass[english,letterpaper]{amsart}
\usepackage[toc,page]{appendix}
\usepackage{amsopn}                    
\usepackage{amsmath} 
\usepackage{amssymb}
\usepackage{stmaryrd}
\usepackage[all,2cell,emtex]{xy}
\usepackage[mathscr]{euscript}
\usepackage{mathrsfs}
\usepackage{bbm}
\DeclareFontFamily{OT1}{pzc}{}
\DeclareFontShape{OT1}{pzc}{m}{it}%
             {<-> s * [1,150] pzcmi7t}{}
\DeclareMathAlphabet{\mathpzc}{OT1}{pzc}%
                                 {m}{it}

\theoremstyle{plain} 
\newtheorem{thm}{Theorem}[section]
\newtheorem{prop}[thm]{Proposition}
\newtheorem{lemma}[thm]{Lemma}
\newtheorem{cor}[thm]{Corollary}

\newtheorem{sslem}{}[thm]
\theoremstyle{remark}
\newtheorem{rem}[thm]{Remark}

\newtheorem{example}[thm]{Example}

\theoremstyle{definition}
\newtheorem{definition}[thm]{Definition}
\newtheorem{paragr}[thm]{}

\theoremstyle{plain} 
\swapnumbers

\numberwithin{equation}{thm}

\renewcommand{\mathcal}{\mathpzc}
\renewcommand{\mathbb}{\mathbbm}

\renewcommand{\leq}{\leqslant}
\renewcommand{\geq}{\geqslant}

\newcommand{\To}{\longrightarrow}
\newcommand{\cats}{\Delta}

\newcommand{\cat}{\mathpzc{Cat}}

\newcommand{\Hom}{\operatorname{\mathrm{Hom}}}
\newcommand{\sHom}{\operatorname{\mathpzc{Hom}}}

\newcommand{\nerf}{N}

\newcommand{\ho}{\operatorname{\mathbf{Ho}}}

\newcommand{\C}{\mathcal{C}}

\newcommand{\I}{\forest}

\newcommand{\T}{\mathcal{T}}

\newcommand{\cornet}{\Lambda}

\newcommand{\bord}{\partial}

\def\TO#1{\mathrel{\hbox to #1pt{\rightarrowfill}}}
\def\OT#1{\mathrel{\hbox to #1pt{\leftarrowfill}}}

\def\limproj{\mathop{\oalign{\rm lim\cr
\hidewidth$\longleftarrow$\hidewidth\cr}}}%
\renewcommand{\varprojlim}{\limproj}%

\newcommand{\Wpr}{\mathsf{W}}

\renewcommand{\to}{\To}
\newcommand{\todouble}{\xymatrixcolsep{1pc}\xymatrix{\ar@<.5ex>[r]\ar@<-.5ex>[r]&}}
\newcommand{\todoubleop}{\xymatrixcolsep{1pc}\xymatrix{\ar@<.5ex>[r]&\ar@<.5ex>[l]}}

\renewcommand{\hookrightarrow}{{\hskip -1.5pt\raise 1.5pt\vbox{\xymatrixcolsep{.9pc}\xymatrix{\ar@{^{(}->}[r]&}}\hskip -3.5pt}}

\newcommand{\intcoin}[5]{\raise 12pt\vbox{\xymatrixcolsep{.9pc}\xymatrixrowsep{.7pc}\xymatrix{%
\scriptstyle #1\ar[r]^{\scriptscriptstyle #5}\ar[d]_{\scriptscriptstyle #4}&\scriptstyle #3\\\scriptstyle #2}}}

\renewcommand{\xrightarrow}[1]{{\hskip -2.5pt\xymatrixcolsep{1.7pc}\xymatrix{\ar[r]^{#1}&}\hskip -2.5pt}}

\newcommand{\sset}{\mathpzc{sSet}}
\newcommand{\dset}{\mathpzc{dSet}}

\newcommand{\trees}{\Omega}
\newcommand{\oper}{\mathpzc{Operad}}
\newcommand{\map}{\mathpzc{hom}}
\newcommand{\join}{\star}
\newcommand{\forest}{\mathcal{T}}
\newcommand{\open}{U}

\title[Dendroidal sets as models for homotopy operads]{Dendroidal sets\\
as models for homotopy operads}

\author[D.-C. Cisinski]{Denis-Charles Cisinski}
\address{LAGA\\
CNRS~(UMR 7539)\\
Universit\'e Paris~13\\
\hbox{Avenue~Jean-Baptiste~Cl\'ement}\\
93430~Villetaneuse\\France}
\email{cisinski@math.univ-paris13.fr}
\urladdr{http://www.math.univ-paris13.fr/~cisinski/}

\author[I. Moerdijk]{Ieke Moerdijk}
\address{Mathematisch Instituut\\
Universiteit Utrecht\\
PO.Box 80.010\\
{3508~TA~Utrecht}\\
The~Netherlands}
\email{moerdijk@math.uu.nl}
\urladdr{http://www.math.uu.nl/people/moerdijk/}


\begin{document}
\begin{abstract}
The homotopy theory of $\infty$-operads is defined
by extending Joyal's homotopy theory of $\infty$-categories
to the category of dendroidal sets. We prove that the
category of dendroidal sets is endowed with a model
category structure whose fibrant objects are the
$\infty$-operads (i.e. dendroidal inner Kan complexes).
This extends the theory of $\infty$-categories in the
sense that the Joyal model category structure on simplicial
sets whose fibrant objects are the $\infty$-categories
is recovered from the model category structure on
dendroidal sets by simply slicing over the point.
\end{abstract}
\maketitle
\tableofcontents
\section*{Introduction}

The notion of dendroidal set is an extension of that of simplicial set, suitable for defining and studying nerves of (coloured) operads in the same way as nerves of categories feature in the theory of simplicial sets. It was introduced by one of the authors and I.~Weiss in \cite{dend1}. As explained in that paper, the category $\dset$ of dendroidal sets carries a symmetric monoidal structure, which is closely related to the Boardman-Vogt tensor product for operads \cite{BV}. There is also a corresponding internal Hom of dendroidal sets. The category of dendroidal sets extends the category $\sset$ of simplicial sets, in the precise sense that there are adjoint functors (left adjoint on the left)
\begin{equation*}
i_!:\sset\rightleftarrows\dset:i^*
\end{equation*}
with good properties. In particular,
the functor $i_!$ is strong monoidal and fully faithful, and identifies $\sset$ with the slice category $\dset/\eta$, where $\eta$ is the unit of the monoidal structure on $\dset$. (In fact, this adjunction is an open embedding of toposes.) 

Using these adjoint functors $i_!$ and $i^*$, we can say more precisely how various constructions and results from the theory of simplicial sets extend to that of dendroidal sets. For example, the nerve functor $\nerf:\cat\To\sset$ and its left adjoint, which we denote by $\tau$,  naturally extend to a pair of adjoint functors
\begin{equation*}
\tau_d:\dset\rightleftarrows\oper:\nerf_d
\end{equation*}
which plays a central role in our work. 

The goal of this paper is to lay the foundations for a homotopy theory of  dendroidal sets and ``$\infty$-operads'' (or ``operads-up-to-homotopy'', or ``quasi-operads'') which extends the simplicial theory of $\infty$-categories (or “quasi-categories”) which has recently been developed by Joyal, Lurie and others. Our main  result is the existence of a Quillen closed model structure on the category of dendroidal sets, having the following properties:
\begin{itemize}
\item[1.] This Quillen model structure on $\dset$ is symmetric monoidal\footnote{This
model category structure is not monoidal; see the Erratum at the end of these notes,
where we explain that this does not affect the main results of this paper
nor of its sequels.
We have chosen not to modify the present article to keep it as close as possible to the
published version.}
(in the sense of \cite{Ho}) and left proper;
\item[2.] The fibrant objects of this model structure are precisely the $\infty$-operads.
\item[3.] The induced model structure on the slice category $\dset/\eta$ is precisely the Joyal model structure on simplicial sets \cite{joytier4,lurie}.
\end{itemize}
The existence of such a model structure was suggested in \cite{dend1}. The $\infty$-operads refered to in 2. are the dendroidal analogues of the $\infty$-categories forming the fibrant objects in the Joyal model structure. They are the dendroidal sets satisfying a lifting condition analogous to the weak Kan condition of Boardman-Vogt, and  were introduced in \cite{dend1,dend2} under the name ``(dendroidal) inner Kan complexes''. The dendroidal nerve of every operad is such an
$\infty$-operad; conversely, intuitively speaking, $\infty$-operads are operads in which the composition of operations is only defined up to homotopy, in a way which is associative up to homotopy. For example, the homotopy coherent nerve of a symmetric monoidal topological category is an $\infty$-operad. The theory of $\infty$-operads contains the theory of $\infty$-categories,
as well as the theory of symmetric monoidal $\infty$-categories and of operads
in them. The theory of $\infty$-operads is also likely to be of use in studying the notion
of $\infty$-category enriched in a symmetric monoidal $\infty$-category (e.g. the various notions
of $A_\infty$-categories, dg~categories, weak $n$-categories).

The proof of our main theorem is based on three sources: First of all, we use the general methods of constructing model structures on presheaf categories developed in \cite{Ci3} (we only use the first
chapter and Section 8.1 of that book, which are both elementary).
Secondly, we use some fundamental properties of dendroidal inner Kan complexes proved in \cite{dend2}. And finally, we use some important notions and results from Joyal's seminal paper \cite{joyal}: namely, the theory of join operations and the notions of left or right fibration
of simplicial sets. Apart from these sources, our proof is entirely self-contained. In particular, we do not use the Joyal model structure in our proof, but instead deduce this model structure as a corollary, as expressed in 3. above.

It is known that there are several (Quillen) equivalent models for $\infty$-categories: one is given by a left Bousfield localisation of the Reedy model structure on simplicial spaces and has as its fibrant objects Rezk's complete Segal spaces; another is given by a Dwyer-Kan style model structure on topological categories established by Bergner, in which all objects are fibrant. The equivalence of these approaches is extensively discussed in Lurie's book \cite{lurie}; see also \cite{bergner,joytier4}. It is natural to ask whether analogous models exist for $\infty$-operads.
In two subsequent papers~\cite{dend4,dend5}, we will show that this is indeed the case. We will prove there that the model structure on dendroidal sets described above is equivalent to a model structure on topological operads in which all objects are fibrant, as well as to a model structure on dendroidal spaces whose fibrant objects are ``dendroidal complete Segal spaces''. The models for
$\infty$-categories just mentioned as well as the equivalences between them will again emerge simply by slicing over suitable unit objects of the respective monoidal structures. Together these model categories fit into a row of Quillen equivalences
$$\xymatrix{
\mathpzc{s}\oper\ar[r]^{\sim}&\dset\ar[r]^{\sim}&\mathpzc{dSpaces}\\
\mathpzc{s}\cat\ar[r]^{\sim}\ar[u]&\sset\ar[r]^{\sim}\ar[u]&\mathpzc{sSpaces}\ar[u]
}$$
in which the vertical arrows are (homotopy) full embeddings.

\smallskip

This paper is organized as follows.
In the first section, we recall the basics about dendroidal sets.
In Section \ref{section2}, we state the main results of this paper:
the existence of a model category structure on the category of
dendroidal sets whose fibrant objects are the $\infty$-operads, as
well as its main properties. In Section \ref{section3}, we construct
this model structure through rather formal arguments. At this stage, it is clear,
by construction, that the fibrant objects are $\infty$-operads, but the
converse is not obvious. Sections \ref{section4} and \ref{section5}
provide the tools to prove that any $\infty$-operad is fibrant, following
the arguments which are known to hold in the case of simplicial sets
for the theory of $\infty$-categories.
More precisely, in Section \ref{section4}, we develop a dendroidal
analog of Joyal's join operations, and prove a generalization
of a theorem of Joyal which ensures a right lifting property
for inner Kan fibrations with respect to certain non-inner horns,
under an additional hypothesis of weak invertibility of some $1$-cells.
In Section \ref{section5}, we construct and examine a subdivision of
cylinders of trees in terms of dendroidal horns.
At last, in Section \ref{section6}, we prove that
any $\infty$-operad is fibrant, and study some of
the good properties of fibrations between $\infty$-operads.
This is done by proving an intermediate result which is important
by itself: a morphism of diagrams in an $\infty$-operad
is weakly invertible if and only if it is locally (i.e. objectwise)
weakly invertible (this is where Sections \ref{section4} and \ref{section5}
have their roles to play).

We also added two appendices, which are independent of the rest of this
paper. In Appendix \ref{appA}, we study the join operations on leaves
(while in Section \ref{section4}, we studied join operations on roots),
and in Appendix \ref{appB}, we study another subdivision of cylinders
of trees. In fact, these appendices can be used to provide another proof
of our main results: Section~\ref{section6} might have been written
using Appendices \ref{appA} and \ref{appB} instead of
Sections \ref{section4} and \ref{section5} respectively, without any changes
(except, sometimes, replacing the evaluation
by $1$ by the evaluation by $0$, whenever necessary).
However, these appendices are not formal consequences of
the rest ot these notes, and it will be useful to have this
kind of results available for further work on the subject.

\section{Dendroidal sets}\label{section1}
\begin{paragr}
Recall from \cite{dend1} the category of trees $\trees$.
The objects of $\trees$ are non-empty non-planar trees with
a designated root, and given two trees $T$ and $T'$, a map from
$T$ to $T'$ is a morphism of the corresponding operads which, in these notes, we will
denote by $T$ and $T'$ again.
Hence, by definition, the category of trees is a full subcategory
of the category of operads.
Recall that the category $\dset$ of \emph{dendroidal sets} is defined as the
category of presheaves of sets on the category of trees $\trees$.
Given a tree $T$, we denote by $\Omega[T]$ the dendroidal set
represented by $T$.

Let $0$ be the tree with only one edge, and set $\eta=\Omega[0]$.
Then the category $\trees/\eta$
identifies canonically with the category $\cats$ of simplices, so that
the category $\dset/\eta$ is canonically equivalent to the category $\sset$
of simplicial sets.
The corresponding functor
\begin{equation}
i:\cats\to\trees \ , \quad [n]\longmapsto i[n]=n
\end{equation}
is fully faithful and its image is a sieve in $\trees$.
This functor $i$ induces an adjunction
\begin{equation}
i_!:\sset\rightleftarrows\dset:i^*
\end{equation}
(where $i_!$ is the left Kan extension of $i$).
Under the identification $\sset=\dset/\eta$, the functor $i_!$
is simply the forgetful functor from $\dset/\eta$ to $\dset$.
The functor $i_!$ is fully faithful and makes $\sset$ into an open
subtopos of $\dset$. In other words, if there is a map of dendroidal sets
$X\To Y$ with $Y$ a simplicial set, then $X$ has to be a simplicial set as well.

We also recall the pairs of adjoint functors
\begin{equation}
\tau:\sset\rightleftarrows\cat:\nerf\quad\text{and}\quad
\tau_d:\dset\rightleftarrows\oper:\nerf_d
\end{equation}
where $\nerf$ and $\nerf_d$ denote the nerve functors from the category
of categories to the category of simplicial sets and from the category
of (symmetric coloured) operads to the category of dendroidal sets.

The category of operads is endowed with a closed symmetric monoidal structure:
the tensor product is defined as the Boardman-Vogt tensor product; see \cite[Section 5]{dend1}.
This defines canonically a unique closed symmetric monoidal structure on the category
of dendroidal sets such that the functor $\tau_d$ is symmetric monoidal, and such that,
for two trees $T$ and $S$, we have
$$\Omega[T]\otimes\Omega[S]=\nerf_d(T\otimes^{}_{BV}S)\, ,$$
where $T\otimes^{}_{BV}S$ is the Boardman-Vogt tensor product of operads.
We will denote internal Hom objects by $\sHom(A,X)$ or by $X^A$.

Note that the functor $i_!:\sset\To\dset$ is a symmetric
monoidal functor, if we consider $\sset$ with its closed cartesian
monoidal structure.

The functor $i^*$ turns the category of dendroidal sets into a
simplicial category; given two dendroidal sets $A$ and $X$,
we will write $\map(A,X)$ for $i^*(\sHom(A,X))$, the simplicial set of maps from
$A$ to $X$.
\end{paragr}

\begin{paragr}
We recall here from \cite{dend1} the different kinds of faces
of trees in $\Omega$.

Let $T$ be a tree.

If $e$ is an inner edge of $T$, we will denote
by $T/e$ the tree obtained from $T$ by contracting $e$.
We then have a canonical inclusion
\begin{equation}\label{definnerface}
\partial_e:T/e\To T\, .
\end{equation}
A map of type \eqref{definnerface} is called an \emph{inner face} of $T$.

If $v$ is a vertex of $T$, with the property
that all but one of the edges incident to $v$ are outer,
we will denote by $T/v$ the tree obtained from $T$
by removing the vertex $v$ and all the outer edges incident to it.
We then have a canonical inclusion
\begin{equation}\label{defouterface}
\partial_v:T/v\To T\, .
\end{equation}
A map of type \eqref{defouterface} is called an \emph{outer face} of $T$.

A map of type \eqref{definnerface} or \eqref{defouterface}
will be called an \emph{elementary face} of $T$.

We define $\bord\Omega[T]$ as the union in $\dset$ of all
the images of elementary face maps $\Omega[T/x]\To\Omega[T]$. We thus have, by
definition, an inclusion
\begin{equation}\label{defboundaryoftree}
\bord\Omega[T]\To\Omega[T]\, .
\end{equation}
Maps of shape \eqref{defboundaryoftree} are called \emph{boundary inclusions}.
The image of a face map $\partial_x$ will sometimes
be denoted by $\partial_x(T)$ for short.

We will call \emph{faces} the maps of $\Omega$ which are obtained, up to an isomorphism,
as compositions of elementary faces. It can be checked that faces are exactly the monomorphisms
in $\Omega$; see \cite[Lemma 3.1]{dend1}.
\end{paragr}

\begin{paragr}
A monomorphism of dendroidal sets $X\To Y$ is \emph{normal}
if for any tree $T$, any non degenerate dendrex $y\in Y(T)$
which does not belong to the image of $X(T)$ has a trivial
stabilizer $\mathrm{Aut}(T)_y\subset\mathrm{Aut}(T)$.
A dendroidal set $X$ is \emph{normal} if the map $\emptyset\To X$ is normal.
For instance, for any tree $T$, the dendroidal set
$\Omega[T]$ is normal.
\end{paragr}

\begin{prop}\label{propertiesnormal}
The class of normal monomorphisms is stable by pushouts, transfinite
compositions and retracts. Furthermore, this is the smallest class
of maps in $\dset$ which is closed under pushouts and tranfinite compositions, and
which contains the boundary inclusions $\bord\Omega[T]\To\Omega[T]$, $T\in\Omega$.
\end{prop}

\begin{proof}
This follows from \cite[Proposition 8.1.35]{Ci3}.
\end{proof}

\begin{prop}\label{caracnormal}
A monomorphism of dendroidal sets $X\To Y$ is normal if and only
if for any tree $T$, the action of $\mathrm{Aut}(T)$ on $Y(T)-X(T)$
is free.
\end{prop}

\begin{proof}
It is easily seen that
the class of monomorphisms which satisfy the above property
is stable by pushouts and transfinite compositions, and
contains the boundary inclusions $\bord\Omega[T]\To\Omega[T]$.
It thus follows from the preceding proposition that
any normal monomorphism has this property.
But it is also obvious that any monomorphism with this property
is normal.
\end{proof}

\begin{cor}\label{normalfree}
A dendroidal set $X$ is normal if and only if for any
tree $T$, the action of the group $\mathrm{Aut}(T)$ on $X(T)$
is free.
\end{cor}

\begin{cor}\label{relativenormal}
Given any map of dendroidal sets $X\To Y$, if $Y$ is normal, then $X$
is normal.
\end{cor}

\begin{cor}\label{relativenormalmono}
Any monomorphism $i:A\To B$ with $B$ normal is
a normal monomorphism.
\end{cor}
%

\begin{prop}\label{BVtensOK}
Let $A\to B$ and $X\To Y$ be two normal monomorphisms.
The induced map
$$A\otimes Y\amalg_{A\otimes X}B\otimes X\To B\otimes Y$$
is a normal monomorphism.
\end{prop}

\begin{proof}
As the class of normal monomorhisms is generated by the
boundary inclusions, it is sufficient to check this property
in this case; see e.g. \cite[Lemma 4.2.4]{Ho}.

Consider now two trees $S$ and $T$. We have to show that the map
$$\bord\Omega[S]\otimes\Omega[T]\amalg_{\bord\Omega[S]\otimes\bord\Omega[T]}
\Omega[S]\otimes\bord\Omega[T]\To\Omega[S]\otimes\Omega[T]$$
is a normal monomorphism. But as $\Omega[S]\otimes\Omega[T]$
is the dendroidal nerve of the Boardman-Vogt tensor product of $S$ and $T$, which
is $\Sigma$-free, it is a normal dendroidal set.
Hence we are reduced to prove that the
above map is a monomorphism. This latter property is equivalent
to the fact that the commutative square
$$\xymatrix{
\bord\Omega[S]\otimes\bord\Omega[T]\ar[r]\ar[d]&\bord\Omega[S]\otimes\Omega[T]\ar[d]\\
\Omega[S]\otimes\bord\Omega[T]\ar[r]&\Omega[S]\otimes\Omega[T]}$$
is a pullback square in which any map is a monomorphism.
As the nerve functor preserves pullbacks, this reduces
to the following property: for any elementary faces
$S/x\To S$ and $T/y\To T$ the commutative square
$$\xymatrix{
S/x\otimes^{}_{BV}T/y\ar[r]\ar[d]&S/x\otimes^{}_{BV}T\ar[d]\\
S\otimes^{}_{BV}T/y\ar[r]&S\otimes^{}_{BV}T
}$$
is a pullback square of monomorphisms in the category of operads.
This is an elementary consequence of the definitions involved.
\end{proof}

\begin{paragr}
Under the assumtions of Proposition \ref{BVtensOK}, we shall write
$A\otimes Y\cup B\otimes X$ instead of
$A\otimes Y\amalg_{A\otimes X}B\otimes X$.
\end{paragr}

\section{Statement of main results}\label{section2}

In this section, we state the main
results of this paper.

\begin{paragr}\label{definnerextan}
Recall from \cite[Section 5]{dend2} the notion of \emph{inner horn}.
Given an inner edge $e$ in a tree $T$, we get
an inclusion
\begin{equation}\label{innerhorns}
\cornet^e[T]\To\Omega[T]\ ,
\end{equation}
where $\cornet^e[T]$ is obtained as the union of all the images
of elementary face maps which are distinct from the face
$\partial_e : T/e\To T$.
The maps of shape \eqref{innerhorns} are called \emph{inner horn inclusions.}

A map of dendroidal sets is called an \emph{inner anodyne extension}
if it belongs to the smallest class of maps which is stable by pushouts,
transfinite composition and retracts, and which contains the inner horn
inclusions.

A map of dendroidal sets is called an \emph{inner Kan fibration}
if it has the right lifting property with respect to the class
of inner anodyne extensions (or, equivalently, to the set of inner horn
inclusions).

A dendroidal set $X$ is an \emph{inner Kan complex} if the map from $X$ to the terminal
dendroidal set is an inner Kan fibration.
We will also call inner Kan complexes \emph{$\infty$-operads}.
For example, for any operad $\mathcal{P}$,
the dendroidal set $\nerf_d(\mathcal{P})$ is an $\infty$-operad; see \cite[Proposition 5.3]{dend2}.
In particular, for any tree $T$, the dendroidal set $\Omega[T]$ is an
$\infty$-operad. For a simplicial set $K$, its image by $i_!$ is
an $\infty$-operad if and only if $K$ is an \emph{$\infty$-category}
(i.e. $K$ is a quasi-category in the sense of \cite{joyal}).

A map of dendroidal sets will be called a \emph{trivial fibration} if it has the right
lifting property with respect to normal monomorphisms.

Note that the small object argument implies that we can factor any map of
dendroidal sets into a normal monomorphism followed by a trivial fibration
(resp. into an inner anodyne extension followed by an inner Kan fibration).
\end{paragr}

\begin{rem}\label{trivnormal}
A morphism between normal dendroidal sets is a trivial fibration if and only
if it has the right lifting property with respect to monomorphisms: this
follows immediately from Corollaries \ref{relativenormal}
and \ref{relativenormalmono}.
\end{rem}

\begin{paragr}
Recall the naive model structure on the category of operads~\cite{weissthesis}:
the weak equivalences are the equivalences of operads, i.e.
the maps $f:\mathcal{P}\To \mathcal{Q}$ which are fully faithful and
essentially surjective: for any $n+1$-uple of objects
$(a_1,\ldots,a_n,a)$ in $\mathcal{P}$, $f$ induces a bijection
$$\mathcal{P}(a_1,\ldots,a_n;a)\To \mathcal{Q}(f(a_1),\ldots,f(a_n);f(a))\, ,$$
and any object of $\mathcal{Q}$ is isomorphic to the image of some
object in $\mathcal{P}$. The fibrations are \emph{operadic fibrations}, i.e.
the maps $f:\mathcal{P}\To \mathcal{Q}$ such that, given any
isomorphism $\beta:b_0\To b_1$ in $\mathcal{Q}$, and any
object $a_1$ in $\mathcal{P}$ such that $f(a_1)=b_1$, there exists
an isomorphism $\alpha:a_0\To a_1$ in $\mathcal{P}$, such that
$f(\alpha)=\beta$.

This model structure is closely related with the naive model structure
on $\cat$ (for which the weak equivalences are the equivalences of categories).
In fact, the latter can be recovered from the one on operads by slicing over
the unit operad (which is also the terminal category).
The fibrations of the naive model structure on $\cat$ will be called the
\emph{categorical fibrations}.
\end{paragr}

\begin{thm}\label{mainthm}
The category of dendroidal sets is endowed with
a model category structure for which the cofibrations
are the normal monomorphisms, the fibrant objects are the
$\infty$-operads, and the fibrations between fibrant objects
are the inner Kan fibrations between $\infty$-operads whose image
by $\tau_d$ is an operadic fibration.
The class of weak equivalences is the smallest
class of maps of dendroidal sets $\Wpr$ which satisfies the
following three properties.
\begin{itemize}
\item[(a)] (`$2$ out $3$ property')
In any commutative triangle, if two maps are in $\Wpr$, then
so is the third.
\item[(b)] Any inner anodyne extension is in $\Wpr$.
\item[(c)] Any trivial fibration between $\infty$-operads
is in $\Wpr$.
\end{itemize}
\end{thm}

\begin{proof}
This follows from Proposition \ref{cmfdendbasic},
Theorem \ref{characfibinftyopermain},
and Corollary \ref{charoperweakequiv}.
\end{proof}

\begin{cor}\label{quillendendopermainthm}
The adjunction $\tau_d:\dset\rightleftarrows\oper:\nerf_d$
is a Quillen pair. Moreover, the two functors $\tau_d$
and $\nerf_d$ both preserve weak equivalences.
In particular, a morphism of operads is an equivalence of
operads if and only if its dendroidal nerve is
a weak equivalence.
\end{cor}

\begin{proof}
See \ref{quillendendoper}.
\end{proof}

\begin{prop}\label{addpropertiescmf}
The model category structure of Theorem \ref{mainthm} has the following
additional properties:
\begin{itemize}
\item[(a)] it is left proper;
\item[(b)] it is cofibrantly generated (it is even combinatorial);
\item[(c)] it is symmetric monoidal.
\end{itemize}
\end{prop}

\begin{proof}
See Propositions \ref{cmfdendbasic} and \ref{Jcmfmonoidal}.
\end{proof}

\begin{cor}\label{cormainthm1}
For any normal dendroidal set $A$ and any $\infty$-operad $X$,
the set of maps $[A,X]=\Hom_{\ho(\dset)}(A,X)$ is canonically
identified with the set of isomorphism classes
of objects in the category $\tau\map(A,X)$.
\end{cor}

\begin{proof}
See Proposition \ref{htpyclassesmaps}.
\end{proof}

\begin{cor}\label{charequiinftyoper1}
Let $f:X\To Y$ be a morphism of $\infty$-operads.
The following conditions are equivalent.
\begin{itemize}
\item[(a)] The map $f:X\To Y$ is a weak equivalence.
\item[(b)] For any normal dendroidal set $A$, the map
$$\tau_d\sHom(A,X)\To\tau_d\sHom(A,Y)$$
is an equivalence of operads.
\item[(c)] For any normal dendroidal set $A$, the map
$$\tau\map(A,X)\To\tau\map(A,Y)$$
is an equivalence of categories.
\end{itemize}
\end{cor}

\begin{proof}
Remember that, by definition (and any $\infty$-operad
being fibrant), the map $f$ is
a weak equivalence if and only if, for any
normal dendroidal set $A$, the induced map
$$[A,X]\To [A,Y]$$
is bijective. This corollary is thus a direct consequence of
Corollaries \ref{quillendendopermainthm} and \ref{cormainthm1}
and of the fact the model category structure on $\dset$ is monoidal.
\end{proof}

\begin{cor}\label{charequiinftyoper2}
Let $u:A\To B$ be a morphism of normal dendroidal sets.
The following conditions are equivalent.
\begin{itemize}
\item[(a)] The map $u:A\To B$ is a weak equivalence.
\item[(b)] For any $\infty$-operad $X$, the map
$$\tau_d\sHom(B,X)\To\tau_d\sHom(A,X)$$
is an equivalence of operads.
\item[(c)] For any $\infty$-operad $X$, the map
$$\tau\map(B,X)\To\tau\map(A,X)$$
is an equivalence of categories.
\end{itemize}
\end{cor}

\begin{proof}
The fibrant objects of $\dset$ are exactly the $\infty$-operads.
Hence, the map $u:A\To B$ is a weak equivalence if and only
if, for any $\infty$-operad $X$, the map
$$[B,X]\To[A,X]$$
is bijective. We conclude the proof using the same arguments as in the proof
of Corollary \ref{charequiinftyoper1}.
\end{proof}

\begin{cor}[Joyal]\label{corjoyalmodcat}
The category of simplicial sets is endowed with
a left proper, cofibrantly generated, symmetric monoidal
model category structure for which the cofibrations
are the monomorphisms, the fibrant objects are the
$\infty$-ca\-te\-go\-ries, and the fibrations between fibrant objects
are the inner Kan fibrations between $\infty$-categories whose image
by $\tau$ is a categorical fibration.
\end{cor}

\begin{proof}
The model category structure on $\dset$ induces
a model category structure on $\dset/\eta\simeq\sset$;
see also Remark \ref{remcmfdendbasic2} for $B=\eta$.
\end{proof}

\begin{rem}
Note that, the functor $i_!:\sset\To\dset$ is
fully faithful and symmetric monoidal. Moreover,
for any simplicial sets $A$ and $X$, we have
$\map(i_!(A),i_!(X))=X^A$. We deduce from this that
the induced map
$$\Hom_{\ho(\sset)}(A,X)\To\Hom_{\ho(\dset)}(i_!(A),i_!(X))$$
is bijective (where $\ho(\sset)$ denotes the homotopy category of
the Joyal model structure, given by Corollary \ref{corjoyalmodcat}).
As a consequence,
we also have formally the simplicial analogs of Corollaries \ref{cormainthm1},
\ref{charequiinftyoper1} and \ref{charequiinftyoper2}.
\end{rem}

\section{Construction of an abstract model category for $\infty$-operads}\label{section3}

This section is devoted to the construction of a model
category structure on $\dset$. The construction is relatively formal
and uses very little of the theory of dendroidal sets.
By definition, we will have that any fibrant object of
this model category is an $\infty$-operad. The proof
of the converse (any $\infty$-operad is fibrant)
is the `raison d'\^etre' of the next sections.

\begin{prop}\label{tensinneranod}
Let $A\To B$ and $X\To Y$ be an inner anodyne extension and a normal
monomorphism respectively.
The induced map
$$A\otimes Y\cup B\otimes X\To B\otimes Y$$
is an inner anodyne extension.
\end{prop}

\begin{proof}
Using \cite[Corollary 1.1.8]{Ci3}, we see that
it is sufficient to check this property when $A\to B$
is an inner horn inclusion and when $X\To Y$ is a boundary inclusion.
This proposition thus follows from \cite[Proposition 9.2]{dend2}.
\end{proof}

\begin{paragr}
We denote by $J$ the nerve of the contractible groupoid with two objects $0$
and $1$ (i.e. $J$ is the nerve of the fundamental groupoid of $\Delta[1]$).
We will write $J_d=i_!(J)$ for the corresponding dendroidal set.

A morphism of dendroidal sets is a \emph{$J$-anodyne extension}
if it belongs to the smallest class of maps which contains
the inner anodyne extensions and the maps
$$\bord\Omega[T]\otimes J_d\cup\Omega[T]\otimes\{e\}\To\Omega[T]\otimes J_d\quad
\text{$T\in\Omega$, $e=0,1$\, ,}$$
and which is closed under pushouts, transfinite compositions and retracts.

A morphism of dendroidal sets will be called a \emph{$J$-fibration}
if it has the right lifting property with respect to
$J$-anodyne extensions.

A dendroidal set $X$ is $J$-fibrant if the map from $X$ to the terminal
dendroidal set is a $J$-fibration.
\end{paragr}

\begin{prop}\label{tensinneranod1bis}
Let $A\To B$ and $X\To Y$ be a $J$-anodyne extension and a normal
monomorphism respectively.
The induced map
$$A\otimes Y\cup B\otimes X\To B\otimes Y$$
is a $J$-anodyne extension.
\end{prop}

\begin{proof}
Using \cite[Corollary 1.1.8]{Ci3},
this follows formally from the definition and from Proposition \ref{tensinneranod}.
\end{proof}

\begin{paragr}\label{cmfovernormal}
Let $B$ be a dendroidal set. Denote by $\mathsf{An}_B$ the class of maps
of $\dset/B$ whose image in $\dset$ is $J$-anodyne.
For each dendroidal set $X$ over $B$, with structural map $a:X\To B$,
we define a cylinder of $X$ over $B$
\begin{equation}\label{cmfovernormal0}
\begin{split}
\xymatrix{
X\amalg X\ar[r]^{(\partial^0_X,\partial^1_X)}\ar[dr]_{(a,a)}
&J_d\otimes X\ar[r]^{\sigma_X}\ar[d]^{a'}&X\ar[ld]^a\\
&B&
}\end{split}
\end{equation}
in which $\partial^e_X$ is the tensor product of $\{e\}\To J_d$ with $1_X$,
while $\sigma_X$ is the tensor product of $J_d\To \eta$ with $1_X$, and $a'$
is the composition of $1_{J_d}\otimes a$ with the map $\sigma_B$.

These cylinders over $B$ define the notion \emph{$J$-homotopy over $B$}
(or \emph{fiberwise $J$-homotopy})
between maps in $\dset/B$. Given two dendroidal sets $A$ and $X$ over $B$,
we define $[A,X]_B$ as the quotient of the set $\Hom_{\dset/B}(A,X)$
by the equivalence relation generated by the relation of $J$-homotopy
over $B$. A morphism $A\To A'$ of dendroidal sets over $B$ is
a \emph{$B$-equivalence} if, for any dendroidal set $X$ over $B$
such that the structural map $X\To B$ is a $J$-fibration, the map
$$[A',X]_B\To [A,X]_B$$
is bijective.

In the case $B$ is normal, any monomorphism over $B$ is normal;
see Corollaries \ref{relativenormal} and \ref{relativenormalmono}.
We see from Proposition \ref{tensinneranod1bis}
and from \cite[Lemma 1.3.52]{Ci3} that the class $\mathsf{An}_B$ is a class of anodyne extensions
with respect to the functorial cylinder \eqref{cmfovernormal0}
in the sense of \cite[Definition 1.3.10]{Ci3}.
In other words, the functorial cylinder \eqref{cmfovernormal0}
and the class $\mathsf{An}_B$ form a \emph{homotopical structure}
on the category $\dset/B$ in the sense of \cite[Definition 1.3.14]{Ci3}.
As a consequence, a direct application of \cite[Theorem 1.3.22,
Proposition 1.3.36 and Lemma 1.3.52]{Ci3}
leads to the following statement\footnote{The results of \cite{Ci3}
are stated for presheaves categories, so that, strictly speaking, to
apply them, we implicitely use the canonical equivalence of
categories between $\dset/B$ and the category of presheaves on $\Omega/B$.}.
\end{paragr}

\begin{prop}\label{cmfovernormal1}
For any normal dendroidal set $B$, the category $\dset/B$
of dendroidal sets over $B$ is endowed with a left proper
cofibrantly generated model category structure for which the
weak equivalences are the $B$-equivalences, the cofibrations are the
monomorphisms, and the fibrant objects are the dendroidal sets $X$ over $B$
such that the structural map is a $J$-fibration. Moreover,
a morphism between fibrant objects is a fibration in $\dset/B$
if and only if its image in $\dset$ is a $J$-fibration.
\end{prop}

\begin{rem}
Any $J$-anodyne extension over $B$ is a trivial cofibration
in the model structure of the preceding proposition; see
\cite[Proposition 1.3.31]{Ci3}.
\end{rem}

\begin{lemma}\label{cmfovernormal2}
Let $p:X\To Y$ be a trivial fibration between normal dendroidal sets.
Any section $s:Y\To X$ is a $J$-anodyne extension.
\end{lemma}

\begin{proof}
This is a particular case of \cite[Corollary 1.3.35]{Ci3}
applied to the homotopical structure defined in \ref{cmfovernormal}
on $\dset/Y$.
\end{proof}

\begin{paragr}
We fix once and for all a normalization $E_\infty$ of the terminal
dendroidal set: i.e., we choose a normal dendroidal set $E_\infty$
such that the map from $E_\infty$ to the terminal dendroidal set is a trivial
fibration.
\end{paragr}

\begin{lemma}\label{cmfovernormal3}
For any normal dendroidal set $X$, and any map $a:X\To E_\infty$,
the map $(a,1_X):X\To E_\infty\times X$ is a $J$-anodyne extension.
\end{lemma}

\begin{proof}
This follows immediately from Lemma \ref{cmfovernormal2} because
$(a,1_X)$ is a section of the projection $X\times E_\infty\To X$, which
is a trivial fibration by definition of $E_\infty$.
\end{proof}

\begin{lemma}\label{cmfovernormal4}
Let $i:A\To B$ be a morphism of normal dendroidal sets,
and $p:X\To Y$ a morphism of dendroidal sets. The map $p$
has the right lifting property with respect to $i$ in $\dset$
if and only if, for any morphism $B\To E_\infty$, the map
$1_{E_\infty}\times p$ has the right lifting property
with respect to $i$ in $\dset/E_\infty$.
\end{lemma}

\begin{proof}
Suppose that $1_{E_\infty}\times p$ has the right lifting property
with respect to $i$, and consider the lifting problem below.
$$\xymatrix{
A\ar[r]^a\ar[d]_i&X\ar[d]^p\\
B\ar[r]_b\ar@{..>}[ru]_\ell&Y}$$
As $B$ is normal, there exists a map $\beta:B\To E_\infty$.
If we write $\alpha=\beta \, i$, we see immediately
that the lifting problem
above is now equivalent to the lifting problem
$$\xymatrix{
A\ar[r]^(.4){(\alpha,a)}\ar[d]_i&E_\infty\times X\ar[d]^{1_{E_\infty}\times p}\\
B\ar[r]_(.4){(\beta,b)}\ar@{..>}[ru]^{(\beta,\ell)}&E_\infty\times Y}$$
and this proves the lemma.
\end{proof}

\begin{paragr}
Given a normal dendroidal set $A$ and a $J$-fibrant
dendroidal set $X$, we denote by $[A,X]$ the quotient
of $\Hom_{\dset}(A,X)$ by the equivalence relation
generated by the $J$-homotopy relation (i.e.,
with the notations of \ref{cmfovernormal},
$[A,X]=[A,X]_e$, where $e$ denotes the terminal dendroidal
set).
\end{paragr}

\begin{prop}\label{cmfdendbasic}
The category of dendroidal sets is endowed with a left proper
cofibrantly generated model category in which the cofibrations
are the normal mono\-mor\-phisms, the fibrant objects are the $J$-fibrant
dendroidal sets, and the fibrations between fibrant objects
are the $J$-fibrations. Furthermore, given a normal
dendroidal set $A$ and a $J$-fibrant dendroidal set $X$, we have
a canonical identification
$$[A,X]=\Hom_{\ho(\dset)}(A,X)\, .$$
\end{prop}

\begin{proof}
Proposition \ref{cmfovernormal1} applied to $B=E_\infty$
gives us a model category structure on $\dset/E_\infty$.
Consider the adjunction
$$p_!:\dset/E_\infty\rightleftarrows\dset:p^*\, ,$$
where $p^*$ is the functor $X\longmapsto E_\infty\times X$.
It follows obviously from Lemma \ref{cmfovernormal3}
that the functor $p^*p_!$ is a left Quillen
equivalence from the category $\dset/E_\infty$ to itself.
This implies immediately that the adjunction $(p_!,p^*)$
satisfies all the necessary hypothesises to define a model
structure on $\dset$ by transfer; see e.g.
\cite{Cr} or \cite[Proposition 1.4.23]{Ci3}.
In other words, the category of dendroidal sets is endowed with
a cofibrantly generated model category structure for which the
weak equivalences (resp. the fibrations) are the maps whose
image by $p^*$ is a weak equivalence (resp. a fibration) in
$\dset/E_\infty$. The description of cofibrations follows from
Proposition \ref{propertiesnormal}.
We know that the fibrations between fibrant objects in $\dset/E_\infty$
are the maps whose image in $\dset$ is a $J$-fibration;
see Proposition \ref{cmfovernormal1}.
The description of fibrant objects and of fibrations between fibrant objects
in $\dset$ as $J$-fibrant objects and $J$-fibrations
is thus a direct consequence of Lemma \ref{cmfovernormal4}. The identification
$[A,X]=\Hom_{\ho(\dset)}(A,X)$ is obtained from the general
description of the set of maps from a cofibrant object to a fibrant
object in an abstract model category. It remains to prove left 
properness: this follows from the left properness of
the model category structure of Proposition \ref{cmfovernormal1}
for $B=E_\infty$ (which is obvious, as any object over $E_\infty$
is cofibrant), and from the fact that $p^*$ preserves cofibrations
as well as colimits, while it preserves and detects weak equivalences.
\end{proof}

\begin{paragr}\label{remcmfdendbasic1}
The weak equivalences of the model structure defined in Proposition
\ref{cmfdendbasic} will be called the \emph{weak operadic
equivalences}.

Given a dendroidal set $A$, a \emph{normalization} of $A$
is a trivial fibration $A'\To A$ with $A'$ normal.
For instance, the projection $E_\infty\times A\To A$ is
a normalization of $A$ (as $E_\infty$ is normal, it follows from
Corollary \ref{relativenormal} that $E_\infty\times A$ is normal).
For a morphism of dendroidal sets $f:A\To B$, the following
conditions are equivalent.
\begin{itemize}
\item[(a)] The map $f$ is a weak operadic equivalence.
\item[(b)] For any commutative square
$$\xymatrix{
A'\ar[r]\ar[d]&A\ar[d]\\
B'\ar[r]&B
}$$
in which the horizontal maps are normalizations, and for any
$J$-fibrant dendroidal set $X$, the map $[B',X]\To [A',X]$
is bijective.
\item[(c)] There exists a commutative square
$$\xymatrix{
A'\ar[r]\ar[d]&A\ar[d]\\
B'\ar[r]&B
}$$
in which the horizontal maps are normalizations
such that, for any $J$-fibrant dendroidal set $X$, the map $[B',X]\To [A',X]$
is bijective.
\end{itemize}
\end{paragr}

\begin{rem}\label{remcmfdendbasic2}
Given a normal $J$-fibrant dendroidal set $B$, the model structure
induced on $\dset/B$ by the model structure of Proposition \ref{cmfdendbasic}
coincide with the model structure of Proposition \ref{cmfovernormal1}
(this follows, for instance, from the fact these model structures
have the same cofibrations and fibrations between fibrant objects).
\end{rem}

\begin{rem}\label{remcmfdendbasic3}
The model category structure of Proposition \ref{cmfdendbasic}
is cofibrantly generated. The generating cofibrations are the
inclusions of shape $\bord\Omega[T]\To\Omega[T]$ for any tree $T$.
We don't know any explicit set of generating trivial cofibrations.
However, we know (from the proof of Proposition \ref{cmfdendbasic})
that there exists a generating set of trivial cofibrations
$\mathcal{J}$ for the model structure on $\dset/E_\infty$, such that
$p_!(\mathcal{J})$ is a generating set of trivial cofibrations of $\dset$.
In particular, there exists a generating set of trivial cofibrations
of $\dset$ which consists of trivial cofibrations between
normal dendroidal sets. Statements about trivial cofibrations
will often be reduced to statements about $J$-anodyne extensions
using the following argument.
\end{rem}

\begin{prop}\label{redtrivcofJanext}
The class of trivial cofibrations between normal dendroidal sets is the
smallest class $\C$ of monomorphisms between normal dendroidal sets which
contains $J$-anodyne extensions, and such that, given any
monomorphisms between normal dendroidal sets
$$\xymatrix{A\ar[r]^i&B\ar[r]^j&C\, ,}$$
if $j$ and $ji$ are in $\C$, so is $i$.
\end{prop}

\begin{proof}
Let $i:A\To B$ be a monomorphism between normal dendroidal sets.
As $B$ is normal, we can choose a map from $B$ to $E_\infty$.
We can then choose a commutative diagram over $E_\infty$
$$\xymatrix{
A\ar[r]^a\ar[d]_i&A'\ar[d]^{i'}\\
B\ar[r]_b&B'
}$$
in which $a$ and $b$ are $J$-anodyne extensions, $A'$
and $B'$ are fibrant in $\dset/E_\infty$, and $i'$
is a monomorphism: this follows, for instance, from the fact that
any $J$-fibrant resolution functor constructed with the
small object argument applied to the generating set of
$J$-anodyne extensions preserves monomorphisms; see \cite[Proposition 1.2.35]{Ci3}.
Applying \cite[Corollary 1.3.35]{Ci3}
to the model structure of Proposition \ref{cmfovernormal1} for $B=E_\infty$,
we see that $i$ is a trivial cofibration if and only
if $i'$ is a $J$-anodyne extension. This proves the proposition.
\end{proof}

\begin{prop}\label{Jcmfmonoidal}
The model category structure on $\dset$ is symmetric monoidal.
\end{prop}

\begin{proof}
As we already know that normal monomorphisms
are well behaved with respect to the tensor product
(Proposition \ref{BVtensOK}) it just remains
to prove that, given a normal monomorphism
$i:A\to B$ and a trivial cofibration
$j:C\To D$, the induced map
$$A\otimes D\cup B\otimes C\To B\otimes D$$
is a trivial cofibration. According to
\cite[Lemma 4.2.4]{Ho}, we can assume that
$i$ is a generating cofibration, and $j$
a generating trivial cofibration. In particular, we can
assume that $i$ and $j$ are monomorphisms between normal
dendroidal sets; see Remark \ref{remcmfdendbasic2}.
It is thus sufficient to prove that, given a normal
dendroidal set $A$, the functor $X\longmapsto A\otimes X$
preserves trivial cofibrations between normal dendroidal
sets. By Proposition \ref{redtrivcofJanext}, it is even
sufficient to prove that tensor product by $A$
preserves $J$-anodyne extensions, which
follows from Proposition \ref{tensinneranod1bis}.
\end{proof}

\section{The join operation on trees}\label{section4}

The aim of this section is to study a dendroidal analog of the
join operations on simplicial sets introduced by Joyal in \cite{joyal}.
We shall prove a generalization of \cite[Theorem 2.2]{joyal};
see Theorem \ref{dendjoyal}.

\begin{paragr}
Let $X$ be a $\infty$-operad. A $1$-simplex of $X$
(i.e. a map $\Delta[1]\To i^*(X)$) will be called
\emph{weakly invertible} if the corresponding morphism
in the category $\tau(i^*(X))$ is an isomorphism.

Note that, for any $\infty$-operad $X$, the category
$\tau(i^*(X))$ is canonically isomorphic to the
category underlying the operad $\tau_d(X)$: this comes
from the explicit description of $\tau(i^*(X)$
given by Boardman and Vogt (see \cite[Proposition 1.2]{joyal})
and from its dendroidal generalization, which
describes $\tau_d(X)$ explicitely; see \cite[Proposition 6.10]{dend2}.
As a consequence, weakly invertible $1$-cells in $X$ can be described
as the maps $i_!\Delta[1]=\Omega[1]\To X$ which induce invertible
morphisms in the underlying category of the operad $\tau_d(X)$.
\end{paragr}

\begin{thm}\label{dendjoyal}
Let $T$ be a tree with at least two vertices as well as a unary vertex $r$ at the root,
and let $p:X\To Y$ be an inner Kan fibration between $\infty$-operads.
Then any solid commutative square of the form
$$\xymatrix{
\Lambda^r[T]\ar[r]^f\ar[d]&X\ar[d]^p\\
\Omega[T]\ar[r]^g\ar@{..>}[ur]^h&Y
}$$
in which $f(r)$ is weakly invertible in $X$
has a diagonal filling $h$.
\end{thm}

\begin{paragr}\label{defjoin}
In order to prove this theorem, we will introduce
join operations on forests.

A \emph{forest} is a finite set of trees (i.e. of objects of $\Omega$).
Given a forest $\forest=(T_1,\ldots,T_k)$, $k\geq 0$,
we write $\forest/\dset$ for the category of dendroidal sets
under the coproduct $\Omega[\forest]=\amalg^k_{i=1} \Omega[T_i]$. The objects of $\forest/\dset$
are thus of shape $(X,x_i)=(X,x_1,\ldots,x_k)$, where $X$ is a dendroidal
set, and $x_i\in X(T_i)$, for $1\leq i\leq k$.
Morphisms $(X,x_i)\To (Y,y_i)$ are maps $f:X\To Y$
such that $f(x_i)=y_i$ for all $i$, $1\leq i\leq k$.

Given an integer $n\geq 0$, we construct the tree $\forest\join n$ by
joining the trees $T_1,\cdots,T_k$ together over a new vertex $v$, and then
grafting the result onto $i[n]$ (i.e. onto $[n]$ viewed as a tree).
\begin{equation}\label{picturejoin}
\begin{split}\xymatrix@R=10pt@C=12pt{
&&&&&&&&\\
&&*=0{\bullet}\ar@{-}[ul]\ar@{-}[ur]\ar@{}[u]|{T_1}
&&*=0{\bullet}\ar@{-}[ul]\ar@{-}[ur]\ar@{}[u]|{T_2}\ar@{}[rrr]|{\cdots\cdots}
&&&*=0{\bullet}\ar@{-}[ul]^{}\ar@{-}[ur]\ar@{}[u]|{T_k}&&\\
&&&&*=0{\bullet}\ar@{-}[u]_{a_2}\ar@{-}[ull]^{a_1}\ar@{-}[urrr]_{a_k}&&&&\\
(T_1,\ldots,T_k)\join n=&&&&*=0{\bullet}\ar@{-}[u]_0^(.8){v}\ar@{-}[d]^1&&&&\\
&&&&*=0{}&&&&\\
&&&&*=0{\bullet}\ar@{-}[d]^n\ar@{}[u]|(.75){\vdots}&&&&\\
&&&&*=0{}&&&&
}\end{split}\end{equation}
We insist that the forest $\forest$ might be empty: for $k=0$, we have
\begin{equation}\label{picturejoinempty}
\begin{split}\xymatrix@R=10pt@C=12pt{
&&*=0{\bullet}&&&&\\
&&*=0{\bullet}\ar@{-}[u]_0^(1.2){v}\ar@{-}[d]^1&&&&\\
(\ )\join n=&&*=0{}&&&&\\
&&*=0{\bullet}\ar@{-}[d]^n\ar@{}[u]|(.75){\vdots}&&&&\\
&&*=0{}&&&&
}\end{split}\end{equation}
As each $T_i$, $1\leq i\leq k$, embeds canonically into $\forest\join n$, we can view
$\Omega[\forest\join n]$ as an object of $\forest/\dset$. One checks that there is
a unique functor
$$\cats\To\Omega\ , \quad [n]\longmapsto\forest\join n$$
such that the inclusions $T_i\To\forest\join n$ are functorial in $T_i$
and such that the canonical inclusion $i[n]\To\forest\join n$ is functorial
in $[n]$. This defines a functor
\begin{equation}\label{defjoin1}
\forest\join(-):\cats\To\I/\dset\, .
\end{equation}
By Kan extension, we obtain a colimit preserving functor
which extends \eqref{defjoin1}:
\begin{equation}\label{defjoin2}
\forest\join(-):\sset\To\I/\dset\, .
\end{equation}
We have $\forest\star\Delta[n]=\Omega[\forest\star n]$.
The functor \eqref{defjoin2} has a right adjoint
\begin{equation}\label{defjoin3}
\forest\backslash(-):\forest/\dset\To\sset\, .
\end{equation}
For a one tree forest $\forest=(T)$, we will simply write
$\forest\star K=T\star K$ and $\forest\backslash X=T\backslash X$
for any simplicial set $K$ and any dendroidal set $X$ under $\Omega[T]$.
Under these conventions, these operations extend the join operations introduced
by Joyal in \cite{joyal} in the sense that we have
the following formulas.
$$\begin{aligned}
i[n]\join i_!(K)&=i_!(\Delta[n]\join K)\\
i[n]\backslash i_!(L)&=i_!(\Delta[n]\backslash L)
\end{aligned}$$
Note that the inclusions $\Omega[n]\To\forest\join\Delta[n]$ in $\dset$
induce a natural projection map
\begin{equation}\label{defjoin4}
\pi_X : \forest\backslash X\To i^*(X)
\end{equation}
for any dendroidal set $X$ under $\forest$.
%
\end{paragr}

\begin{rem}
Note that any tree with at least one vertex
$T$ is obtained by joining a forest with an ordinal,
i.e. as $T=\forest\join n$ for some forest $\forest$ and some integer $n\geq 0$.
A tree $T$ has at least two vertices and a unary vertex at the root (as in the
statement of Theorem \ref{dendjoyal}) if and only if
there exists a forest $\forest$ such that $T=\forest\join 1$.
\end{rem}

\begin{paragr}\label{defadmedges}
In order to prove Theorem \ref{dendjoyal}, we will have also
to consider some specific maps of forests. For this purpose,
we introduce the following terminology.

Let $T$ be a tree. A set $A$ of edges in $T$ is
is called \emph{admissible} if, for any input edge $e$ of $T$, and
any vertex $v$ in $T$, if $A$ contains a path (branch) from $e$ to $v$,
then $A$ contains all the edges above $v$.

If $A$ is an admissible set of edges in $T$, we will define a forest $\partial_A(T)$, and
for each tree $S$ in $\partial_A(T)$, a face map $S\To T$ in the category $\Omega$.
Roughly speaking, one deletes from $T$ all edges in $A$, and defines $\partial_A(T)$
as the resulting connected components. A formal definition is by induction on the
cardinality of $A$.
\begin{itemize}
\item[(i)] If $A$ is empty, then $\partial_A(T)=T$.
\item[(ii)] If $A$ contains the root edge $e$ of $T$, let
$T_1,\ldots,T_k$ be the trees obtained from $T$ by deleting
$e$ and the vertex immediately above it, let $A_i=T_i\cap A$,
and define $\partial_A(T)$ as the union of the forests $\partial_{A_i}(T_i)$, $1\leq i\leq k$.
\item[(iii)] If $A$ contains an input edge $a$ of $T$, it must contain all the edges
above the vertex $v$ just below $a$. Let $T_{(v)}$ be the tree obtained from $T$
by pruning away $v$ and all the edges above it. Let $A_{(v)}=T_{(v)}\cap A$, and
define $\partial_A(T)=\partial_{A_{(v)}}(T_{(v)})$.
\item[(iv)] If $A$ contains an inner edge $a$ of $T$, let $T/a$ be the tree obtained
by contracting $a$, and define $\partial_A(T)$ to be
$\partial_{A-\{a\}}(T/a)$.
\end{itemize}
One can check that the steps (i)--(iv) can be performed
in any order, so that the forest $\partial_A(T)$ is well defined.
Each tree $S$ in this forest $\partial_A(T)$ is a face of $T$, hence comes
with a canonical map $S\To T$.
\end{paragr}

\begin{example}
The tree
$$
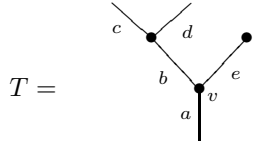
\xymatrix@R=10pt@C=12pt{
&&&&&\\
&&*=0{\bullet}\ar@{-}[ul]^{c}\ar@{-}[ur]_{d}&&*=0{\bullet}&\\
T=&&&*=0{\bullet}\ar@{-}[ul]^{b}\ar@{-}[ur]_{e}&&\\
&&&*=0{}\ar@{-}[u]^a_(.85)v&&
}
$$
has two input edges $c$ and $d$. The edges $b$ and $c$ form
a path from $c$ down to $v$. So any admissible set $A$
which contains $b$ and $c$, for example, must contain $d$ and $e$
as well.
\end{example}

\begin{paragr}
This construction extends to forests in the following way.
Let $\forest=(T_1,\ldots,T_k)$ be a forest. An \emph{admissible subset of edges} $A$ in $\forest$
is a $k$-tuple $A=(A_1,\ldots,A_k)$, where $A_i$ is an admissible set
of edges of $T_i$ for $1\leq i\leq k$.
We can then define the forest $\partial_A(T)$ as the union of the
forests $\partial_{A_i}(T_i)$.
Given any integer $n\geq 0$, we have a canonical map
\begin{equation}\label{defadmedges0}
\partial_A(\forest)\join n\To \forest\join n
\end{equation}
which is characterized by the fact that, given any tree $S$ in
some $\partial_{A_i}(T_i)$, for $1\leq i\leq k$, the diagram
\begin{equation}\label{defadmedges00}
\begin{split}\xymatrix{
S\join n\ar[d]\ar[r]&T_i\join n\ar[d]\\
\partial_A(\forest)\join n\ar[r]&\forest\join n
}\end{split}\end{equation}
commutes. The map \eqref{defadmedges0}
is a monomorphism of trees in $\Omega$ and is
natural in $[n]$ (as an object of $\cats$).
More generally, given an inclusion $A\subset B$ between admissible subsets of edges in $\forest$,
we have canonical monomorphisms of trees
\begin{equation}\label{defadmedges1}
\partial_B(\forest)\join n\To \partial_A(\forest)\join n
\end{equation}
(which is just another instance of \eqref{defadmedges0}
for the forest $\partial_A(\forest)$ with admissible subset of edges given by the
sets $B_i\cap\partial_{A_i}(T_i)$).
The maps \eqref{defadmedges1} define a contravariant functor from
the set of admissible subsets of edges in $\forest$ (partially
ordered by inclusion)
to $\Omega$. Given an inclusion $A\subset B$ of admissible subsets of edges in $\forest$,
there exists a unique morphism
\begin{equation}\label{defadmedges2}
\Omega[\partial_B(\forest)]\To\Omega[\partial_A(\forest)]
\end{equation}
such that the following diagram commutes for any
simplicial set $K$.
\begin{equation}\label{defadmedges3}\begin{split}\xymatrix{
\Omega[\partial_B(\forest)]\ar[r]\ar[d]&\Omega[\partial_A(\forest)]\ar[d]\\
\partial_B(\forest)\join K\ar[r]&\partial_A(\forest)\join K
}\end{split}\end{equation}
By adjunction, we also have natural morphisms
\begin{equation}\label{defadmedges5}
\partial_B(\forest)\backslash X\To\partial_A(\forest)\backslash X
\end{equation}
for all dendroidal sets $X$ under $\Omega[\partial_A(\forest)]$.
\end{paragr}

\begin{example}
If $a$ is the root of $T$, and if $T$ is obtained by grafting
trees $T_i$ with root edges $a_i$ onto a corolla, then the map
of type \eqref{defadmedges1} for $A=\varnothing$ and $B=\{a\}$ is the map
$\partial_a$ given by contracting $a$:
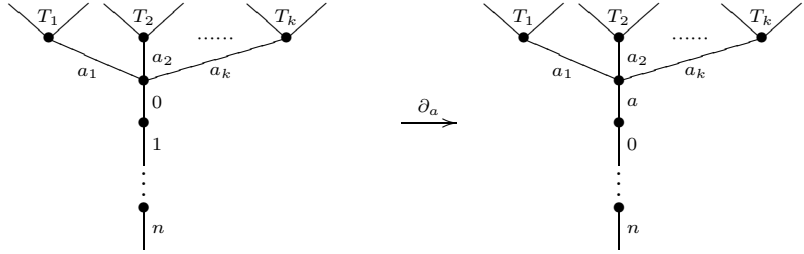
\begin{equation*}
\begin{split}\xymatrix@R=10pt@C=12pt{
&&&&&&&\\
&*=0{\bullet}\ar@{-}[ul]\ar@{-}[ur]\ar@{}[u]|{T_1}
&&*=0{\bullet}\ar@{-}[ul]\ar@{-}[ur]\ar@{}[u]|{T_2}\ar@{}[rrr]|{\cdots\cdots}
&&&*=0{\bullet}\ar@{-}[ul]^{}\ar@{-}[ur]\ar@{}[u]|{T_k}&&\\
&&&*=0{\bullet}\ar@{-}[u]_{a_2}\ar@{-}[ull]^{a_1}\ar@{-}[urrr]_{a_k}&&&&\\
&&&*=0{\bullet}\ar@{-}[u]_0\ar@{-}[d]^1&&&&\\
&&&*=0{}&&&&\\
&&&*=0{\bullet}\ar@{-}[d]^n\ar@{}[u]|(.75){\vdots}&&&&\\
&&&*=0{}&&&&
}\end{split}\xrightarrow{\quad \partial_a\quad}
\begin{split}\xymatrix@R=10pt@C=12pt{
&&&&&&&\\
&*=0{\bullet}\ar@{-}[ul]\ar@{-}[ur]\ar@{}[u]|{T_1}
&&*=0{\bullet}\ar@{-}[ul]\ar@{-}[ur]\ar@{}[u]|{T_2}\ar@{}[rrr]|{\cdots\cdots}
&&&*=0{\bullet}\ar@{-}[ul]^{}\ar@{-}[ur]\ar@{}[u]|{T_k}&&\\
&&&*=0{\bullet}\ar@{-}[u]_{a_2}\ar@{-}[ull]^{a_1}\ar@{-}[urrr]_{a_k}&&&&\\
&&&*=0{\bullet}\ar@{-}[u]_a\ar@{-}[d]^0&&&&\\
&&&*=0{}&&&&\\
&&&*=0{\bullet}\ar@{-}[d]^n\ar@{}[u]|(.75){\vdots}&&&&\\
&&&*=0{}&&&&
}\end{split}
\end{equation*}
\end{example}

\begin{paragr}\label{suffinneranod00}
We will now study an elementary combinatorial situation which we will
have to consider twice to prove Theorem \ref{dendjoyal}:
in the proof of Proposition \ref{adjjoinleftfib0}
and in the proof of \ref{proofofiib}.

Consider a tree $T$. Assume that $T=\forest\join n$
for a non-empty forest $\forest=(T_1,\ldots,T_k)$ and an ordinal $[n]$, $n>0$.

Let $i$, $0\leq i<n$, be an integer, and $\{A_1,\ldots,A_s\}$, $s\geq 1$,
a finite family of admissible subsets of edges in $\T$. Define
$$C\subset D\subset\Omega[T]$$
by
$$C=
\big(\bigcup^s_{r=1}\partial_{A_r}(\forest)\join\Lambda^i[n]\big)\cup\Omega[n]
\quad\text{and}\quad
D=\bigcup^s_{r=1}\partial_{A_r}(\forest)\join\Delta[n]\, ,$$
where $\Omega[n]$ is considered as a subcomplex of $\Omega[T]$
through the canonical map.
\end{paragr}

\begin{lemma}\label{suffinneranod}
Under the assumptions of \ref{suffinneranod00},
the map $C\To D$ is an inner anodyne extension.
\end{lemma}

\begin{proof}
If $\forest$ is the empty forest, we must have $s=1$ and
$A_1=\varnothing$, so that $D=\Omega[T]$, and
$C=\Lambda^i[T]$ is an inner horn. From now on, we will assume
that $\forest$ is non-empty.

Given a forest $\forest'$, the number of edges in $\forest'$
is simply defined as the sum of the number of edges in each of the trees
which occur in $\forest'$.
For each integer $p\geq 0$, write $\mathcal{F}_p$
for the set of faces $F$ which belong to $D$ but not to $C$,
and which are of shape
$F=\Omega[\partial_A(\forest)\join n]$ for an admissible subset of edges $A$
in $\forest$, such that $\partial_A(\forest)$ has exactly $p$ edges.

Define a filtration
$$C=C_0\subset C_1\subset\ldots\subset C_p\subset\ldots\subset D$$
by
$$C_p=C_{p-1}\cup\bigcup_{F\in\mathcal{F}_p} F\ , \quad p\geq 1\, .$$
We have $D=C_p$ for $p$ big enough, and it is sufficient to prove
that the inclusions $C_{p-1}\To C_p$ are inner anodyne for $p\geq 1$.
If $F$ and $F'$ are in $\mathcal{F}_p$, then $F\cap F'$
is in $C_{p-1}$. Moreover, if $F=\Omega[\partial_A(\forest)\join n]$
for an admissible subset of edges $A$, then we have
$$F\cap C_{p-1}=\Lambda^i[\partial_A(\forest)\join n]\, ,$$
which is an inner horn.
Hence we can describe the inclusion $C_{p-1}\To C_p$
as a finite composition of pushouts by inner horn inclusions
of shape $F\cap C_{p-1}\To F$ for $F\in\mathcal{F}_p$.
\end{proof}

\begin{prop}\label{adjjoinleftfib0}
Let $\forest=(T_1,\ldots,T_k)$ be a forest,
and $n\geq 1$, $0\leq i<n$, be integers.
The inclusion $(\forest\join\Lambda^i[n])\cup\Omega[n]\To\forest\join\Delta[n]$
is an inner anodyne extension.
\end{prop}

\begin{proof}
This is a particular case of the preceding lemma.
\end{proof}

\begin{paragr}
Remember from \cite{joyal} that a morphism of
simplicial sets is called a \emph{left} (resp. \emph{right})
\emph{fibration} if it has the right lifting property with respect
to inclusions of shape $\Lambda^i[n]\To\Delta[n]$ for $n\geq 1$
and $0\leq i<n$ (resp. $0<i\leq n$).

A morphism between $\infty$-categories $X\To Y$ is \emph{conservative}
if the induced functor $\tau(X)\To\tau(Y)$ is conservative
(which can be reformulated by saying that a $1$-simplex of $X$
is weakly invertible if and only if its image in $Y$ is
weakly invertible). For instance, by virtue of
\cite[Proposition 2.7]{joyal}, any left (resp. right) fibration
between $\infty$-categories is conservative.
\end{paragr}

\begin{prop}\label{adjjoinleftfib}
Let $p:X\To Y$ be an inner Kan fibration of
dendroidal sets under a forest $\forest$. The map
$$\forest\backslash X\To\I\backslash Y\times_{i^*(Y)}i^*(X)$$
is a left fibration.

In particular, for any $\infty$-operad $X$
under a forest $\forest$, the map $\forest\backslash X\To i^*(X)$ is
a left fibration.
\end{prop}

\begin{proof}
This follows immediately from Proposition \ref{adjjoinleftfib0}
by a standard adjunction argument.
\end{proof}

\begin{cor}\label{adjjoininftycat}
For any $\infty$-operad $X$ and any forest $\forest$ over $X$,
the simplicial set $\forest\backslash X$ is an $\infty$-category.
Similarly, for any inner Kan fibration between
$\infty$-operads $X\To Y$ and any forest $\forest$ over $X$,
the simplicial set $\I\backslash Y\times_{i^*(Y)}i^*(X)$
is an $\infty$-category.
\end{cor}

\begin{proof}
If $X$ is an $\infty$-operad, then $i^*(X)$ is
clearly an $\infty$-category. Since the projection
$\forest\backslash X\To i^*(X)$ is
a left fibration, this implies this corollary.
\end{proof}

As a warm up to prove Theorem \ref{dendjoyal},
we shall consider a particular case.

\begin{lemma}\label{joyaldendemptyforest}
Theorem \ref{dendjoyal} is true if $T=(\ )\join 1$
(where $(\ )$ denotes the empty forest).
\end{lemma}

\begin{proof}
In this case, $T$ is a tree of shape
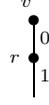
\[
\xy<0cm,0.5cm>:
(0,0)*=0{}="1"+(0.55,-0.3)*{\scriptstyle 1};
(1,0)*=0{\bullet}="2"+(0,0.5)*{\scriptstyle r}+(0.55,-0.8)*{\scriptstyle 0};
(2,0)*=0{\bullet}="3"-(0.2,0.3)+(0.7,0.5)*{\scriptstyle v};
"1";"2" **\dir{-};
"2";"3" **\dir{-};
\endxy
\]
and $\Lambda^r[T]$ is the union of the two faces
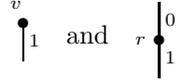
\begin{equation*}\begin{split}
\xy<0cm,0.5cm>:
(1,0)*=0{}="1"+(0,0.5)*{}+(0.55,-0.8)*{\scriptstyle 1};
(2,0)*=0{\bullet}="2"-(0.2,0.3)+(0.7,0.5)*{\scriptstyle v};
"1";"2" **\dir{-};
\endxy
\end{split}\quad\text{and}\quad\begin{split}
\xy<0cm,0.5cm>:
(0,0)*=0{}="1"+(0.55,-0.3)*{\scriptstyle 1};
(1,0)*=0{\bullet}="2"+(0,0.5)*{\scriptstyle r}+(0.55,-0.8)*{\scriptstyle 0};
(2,0)*=0{}="3"-(0.2,0.3)+(0.7,0.5)*{};
"1";"2" **\dir{-};
"2";"3" **\dir{-};
\endxy
\end{split}\end{equation*}
In other words, we get $\Lambda^r[T]=(\ )\join\Lambda^1[1]\cup\Omega[1]$.
Thus, a lifting problem of shape
$$\xymatrix{
\Lambda^r[T]\ar[r]^f\ar[d]&X\ar[d]^p\\
\Omega[T]\ar[r]^g\ar@{..>}[ur]^h&Y
}$$
is equivalent to a lifting problem of shape
$$\xymatrix{
{\{1\}}\ar[r]^{\tilde f}\ar[d]_{}&(\ )\backslash X\ar[d]^\varphi\\
\Delta[1]\ar[r]^(.35){\tilde g}\ar@{..>}[ur]^{\tilde h}&(\ )\backslash Y\times_{i^*(Y)}i^*(X)
}$$
By virtue of Proposition \ref{adjjoinleftfib}, the map
$\varphi$ is a left fibration, and, as left fibrations are stable
by pullback and by composition, so is the projection
of $(\ )\backslash Y\times_{i^*(Y)}i^*(X)$ to
$i^*(X)$. The image of $\tilde g$ by the latter
is nothing but $f(r)$, and, as we know that left fibrations
between $\infty$-categories
are conservative (see \cite[Proposition 2.7]{joyal}),
the $1$-cell $\tilde g$ is quasi-invertible in $(\ )\backslash Y\times_{i^*(Y)}i^*(X)$.
We conclude the proof using \cite[Propositions 2.4 and 2.7]{joyal}.
\end{proof}

\begin{proof}[Proof of Theorem \ref{dendjoyal}]
Let $T$ be a tree with at least two vertices and a unary vertex $r$ at the root.
There is a forest $\forest=(T_1,\cdots,T_k)$, $k\geq 0$, such that
$T=\forest\join 1$. By virtue of Lemma \ref{joyaldendemptyforest},
we may assume that $\forest$ is not the empty forest, or, equivalently, that
$k\geq 1$. We will write $T'=\forest\join 0$. The trees $T$ and $T'$
can be represented as follows.

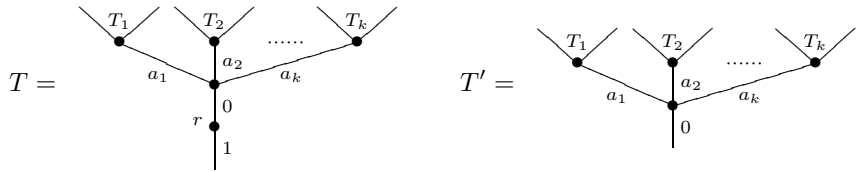
\begin{equation*}
\begin{split}\phantom{a}\end{split}
T=
\begin{split}\xymatrix@R=10pt@C=12pt{
&&&&&&&\\
&*=0{\bullet}\ar@{-}[ul]\ar@{-}[ur]\ar@{}[u]|{T_1}
&&*=0{\bullet}\ar@{-}[ul]\ar@{-}[ur]\ar@{}[u]|{T_2}\ar@{}[rrr]|{\cdots\cdots}
&&&*=0{\bullet}\ar@{-}[ul]^{}\ar@{-}[ur]\ar@{}[u]|{T_k}&&\\
&&&*=0{\bullet}\ar@{-}[u]_{a_2}\ar@{-}[ull]^{a_1}\ar@{-}[urrr]_{a_k}&&&&\\
&&&*=0{\bullet}\ar@{-}[u]_0\ar@{-}[d]^1_(-.15){r\, }&&&&\\
&&&*=0{}&&&&
}\end{split}
T'=
\begin{split}\xymatrix@R=10pt@C=12pt{
&&&&&&&\\
&*=0{\bullet}\ar@{-}[ul]\ar@{-}[ur]\ar@{}[u]|{T_1}
&&*=0{\bullet}\ar@{-}[ul]\ar@{-}[ur]\ar@{}[u]|{T_2}\ar@{}[rrr]|{\cdots\cdots}
&&&*=0{\bullet}\ar@{-}[ul]^{}\ar@{-}[ur]\ar@{}[u]|{T_k}&&\\
&&&*=0{\bullet}\ar@{-}[u]_{a_2}\ar@{-}[ull]^{a_1}\ar@{-}[urrr]_{a_k}&&&&\\
&&&*=0{}\ar@{-}[u]_0
}\end{split}
\end{equation*}

Given a dendroidal set $X$, a map
$\Lambda^r[T]\To X$ corresponds to a compatible
family of maps of simplicial sets
$$\{1\}=\Delta[0]\To\partial_A(T')\backslash X\, ,$$
indexed by the \emph{non-empty} admissible subset of edges $A$ in $T'$.
This family corresponds to a map
$$\Delta[0]\To\varprojlim_A\partial_A(T')\backslash X\, .$$
By separating the case $A=\{0\}$ (the root edge of $T'$) from the others,
the map $\Lambda^r[T]\To X$ corresponds to a commutative
square of shape
$$\xymatrix{
\Delta[0]\ar[r]\ar[d]_{{\partial_0}}&\forest\backslash X\ar[d]\\
\Delta[1]\ar[r]&{\varprojlim\partial_{\bar B}(\forest)\backslash X}
}$$
in which the limit $\varprojlim\partial_{\bar B}(\forest)\backslash X$ is over
the non-empty admissible subsets of edges $B$ in $T'$ with $0\notin B$,
and $\bar B=(B\cap T_1,\ldots,B\cap T_k)$.

Consider from now on an inner Kan fibration between $\infty$-operads $p:X\To Y$.
Lifting problems of shape
$$\xymatrix{
\Lambda^r[T]\ar[r]^f\ar[d]&X\ar[d]^p\\
\Omega[T]\ar[r]^g\ar@{..>}[ur]^h&Y
}$$
correspond to lifting problems
$$\xymatrix{
\Delta[0]\ar[r]^{\tilde f}\ar[d]_{{\partial_0}}&P\ar[d]^\varphi\\
\Delta[1]\ar[r]^{\tilde g}\ar@{..>}[ur]^{\tilde h}&Q
}$$
where $P=\forest\backslash X$ and $Q=U\times_WV$, with
$$U={\varprojlim\partial_{\bar B}(\forest)\backslash X}\, ,
\quad V=\forest\backslash Y\, ,\quad
W={\varprojlim\partial_{\bar B}(\forest)\backslash Y}\, .$$
Exactly like in the proof of Proposition \ref{joyaldendemptyforest},
it now suffices to prove the following three statements:
\begin{itemize}
\item[(i)] the map $\varphi:P\To Q$ is a left fibration;
\item[(ii)] the simplicial set $Q$ is an $\infty$-category;
\item[(iii)] if $f(r)$ is weakly invertible in $X$, then
the $1$-cell $\tilde g$ is weakly invertible in $i^*(Q)$.
\end{itemize}
Note that, as left fibrations are conservative and are stable
by pullback and composition, statements (ii) and (iii)
will follow from the following two assertions:
\begin{itemize}
\item[(iv)] the map $V\To W$ is a left fibration;
\item[(v)] the map $U\To i^*(X)$ is a left fibration.
\end{itemize}
But (iv) is a particular case of (i): just replace $p$
by the map from $Y$ to the terminal dendroidal set.
It thus remains to prove (i) and (v).

\begin{sslem} Proof of \emph{(i)}.\end{sslem}
For $0\leq i<n$, a lifting problem of the form
$$\xymatrix{
\Lambda^i[n]\ar[r]^{}\ar[d]_{}&P\ar[d]^\varphi\\
\Delta[n]\ar[r]^{}\ar@{..>}[ur]^{}&Q
}$$
correspond to a lifting problem of the form
$$\xymatrix{
\Lambda^i[\forest\join n]\ar[r]\ar[d]&X\ar[d]^p\\
\Omega[\forest\join n]\ar[r]\ar@{..>}[ur]&Y
}$$
As $\Lambda^i[\forest\join n]$ is an inner horn, (i) thus follows
from the fact $p$ is an inner Kan fibration.
\begin{sslem}\label{proofofiib} Proof of \emph{(v)}.\end{sslem}
For $0\leq i<n$, a lifting problem of the form
$$\xymatrix{
\Lambda^i[n]\ar[r]^{}\ar[d]_{}&U\ar[d]\\
\Delta[n]\ar[r]^{}\ar@{..>}[ur]^{}&i^*(X)
}$$
corresponds to a lifting problem of the form
$$\xymatrix{
C\ar[r]\ar[d]&X\\
D\ar@{..>}[ur]&
}$$
where the inclusion $C\To D$ can be described as follows.
The dendroidal set $D$ is the union of all the faces $\partial_x(\forest\join n)$
given by contracting an inner edge or a root edge
in one of the trees $T_i$, or by deleting a top
vertex in the tree $\forest\join n$. The dendroidal set $C$
is the union of the image of $\Omega[n]\To\Omega[\forest\join n]$
and all the `codimension $2$' faces of $\Omega[\forest\join n]$
of shape $\partial_j\partial_x(\forest\join n)$, where $\partial_x$
is as above, and $0\leq j\leq n$ is distinct from $i$.
It is now sufficient to check that the inclusion $C\To D$
is an inner anodyne extension, which follows from a
straightforward application of Lemma \ref{suffinneranod}.
\end{proof}

\section{Subdivision of cylinders}\label{section5}

\begin{paragr}\label{prelsubcyl}
Let $S$ be a tree with at least one vertex, and consider the
tensor product $\Omega[S]\otimes\Delta[1]$. It has a subobject
$$A_0=\bord\Omega[S]\otimes\Delta[1]\cup\Omega[S]\otimes\{1\}$$
where $\{1\}\To\Delta[1]$ is $\partial_0:\Delta[0]\To\Delta[1]$.
In this section, we will prove the following result.
\end{paragr}

\begin{thm}\label{thmsubcyl}
There exists a filtration of $\Omega[S]\otimes\Delta[1]$ of the
form
$$A_0\subset A_1\subset\cdots\subset A_{N-1}\subset A_N=\Omega[S]\otimes\Delta[1]\, ,$$
where:
\begin{itemize}
\item[(i)] the inclusion $A_i\To A_{i+1}$ is inner anodyne
for $0\leq i<N-1$;
\item[(ii)] the inclusion $A_{N-1}\To A_N$ fits into a
pushout of the form
$$\xymatrix{
\Lambda^r[T]\ar[r]\ar[d]&A_{N-1}\ar[d]\\
\Omega[T]\ar[r]&A_N
}$$
for a tree $T$ with at least two vertices and a unary vertex $r$ at the root;
\item[(iii)] the map
$$\Delta[1]\To\Lambda^r[T]\To A_{N-1}\subset\Omega[S]\otimes\Delta[1]$$
corresponding to the vertex $r$ in \emph{(ii)} coincides with the inclusion
$$\{e_S\}\otimes\Delta[1]\To\Omega[S]\otimes\Delta[1]$$
where $e_S$ is the edge at the root of the tree $S$.
\end{itemize}
\end{thm}

\begin{paragr}\label{recallsubdivision}
The proof of Theorem \ref{thmsubcyl} is in fact very similar
to that of \cite[Proposition 9.2]{dend2}, stated here
as Proposition \ref{tensinneranod}. We recall from \emph{loc. cit.}
that, for any two trees $S$ and $T$, one can write
$$\Omega[S]\otimes\Omega[T]=\bigcup^N_{i=1}\Omega[T_i]\, ,$$
where $\Omega[T_i]\To\Omega[S]\otimes\Omega[T]$ are `percolation schemes'.
Drawing vertices of $S$ as white, and those of $T$ as black,
these percolation schemes are partially ordered in a natural way, starting
with the tree obtained by stacking a copy of the black tree $T$
on top of each input edge of the white tree $S$, and ending
with the tree obtained by stacking copies of $S$ on top of $T$.
The intermediate trees are obtained by letting the black vertices of $T$
percolate through the white tree $S$, by successive `moves' of the form
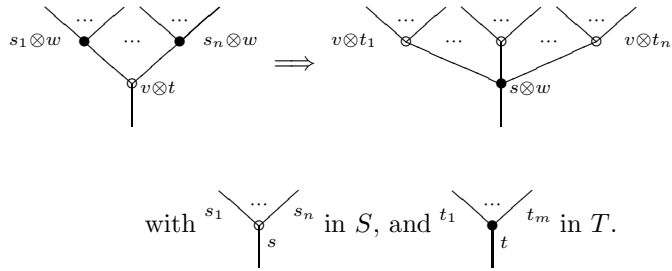
\begin{equation*}
\begin{split}\xymatrix@R=10pt@C=12pt{
&&&&\\
&*=0{\bullet}\ar@{-}[ul]\ar@{-}[ur]\ar@{}[u]|{\cdots}
\ar@{}[rr]|{\cdots}&&*=0{\bullet}\ar@{-}[ul]\ar@{-}[ur]\ar@{}[u]|{\cdots}
\ar@{-}[ur]&\\
&&*=0{\circ}\ar@{-}[ur]_(1.4){s_n\otimes w}\ar@{-}[ul]^(1.4){s_1\otimes w}&&\\
&&*=0{}\ar@{-}[u]_(.9){v\otimes t}&&
}\end{split}\Longrightarrow
\begin{split}\xymatrix@R=10pt@C=12pt{
&&&&&&&\\
&*=0{\circ}\ar@{-}[ul]\ar@{-}[ur]\ar@{}[u]|{\cdots}\ar@{}[rr]|{\cdots}
&&*=0{\circ}\ar@{-}[ul]\ar@{-}[ur]\ar@{}[u]|{\cdots}\ar@{}[rr]|{\cdots}
&&*=0{\circ}\ar@{-}[ul]^{}\ar@{-}[ur]\ar@{}[u]|{\cdots}&&&\\
&&&*=0{\bullet}\ar@{-}[u]_{}\ar@{-}[ull]^(1.4){v\otimes t_1}\ar@{-}[urr]_(1.4){v\otimes t_n}&&&&\\
&&&*=0{}\ar@{-}[u]_(.9){s\otimes w}&&&&
}\end{split}
\end{equation*}
\begin{equation*}
\begin{split}
\phantom{a}
\end{split}
\text{with}
\begin{split}
\xymatrix@R=10pt@C=12pt{
&&\\
&*=0{\circ}\ar@{-}[ur]_(0.6){s_n}\ar@{-}[ul]^(0.6){s_1}
\ar@{}[u]|{\cdots}&\\
&*=0{}\ar@{-}[u]_(.6){s}&
}
\end{split}\text{in $S$, and}
\begin{split}
\xymatrix@R=10pt@C=12pt{
&&\\
&*=0{\bullet}\ar@{-}[ur]_(0.6){t_m}\ar@{-}[ul]^(0.6){t_1}
\ar@{}[u]|{\cdots}&\\
&*=0{}\ar@{-}[u]_(.6){t}&
}
\end{split}\text{in $T$.}
\end{equation*}
In the special case where $T=[1]$, the filtration referred to in Theorem \ref{thmsubcyl}
is given by
$$A_i=A_0\cup\Omega[T_1]\cup\cdots\cup\Omega[T_i]\, ,$$
where $T_1,\ldots,T_N$ is any linear order on the percolation schemes
extending the natural partial order.
\end{paragr}

\begin{rem}\label{remendfiltSotimes1}
For any tree $S$ with at least one vertex, and root edge named $e_S$ ($e$ for
`exit'), the last tree $T_N$ in the partial order of percolation schemes
for $\Omega[S]\otimes\Delta[1]$ is of shape
\begin{equation*}
\begin{split}\xymatrix@R=10pt@C=12pt{
&&&\\
&&*=0{\circ}\ar@{-}[ur]\ar@{-}[ul]\ar@{}[u]|{S}&\\
T_N=&&*=0{\bullet}\ar@{-}[u]^(.005){r\ }_{(e_S,0)}\ar@{-}[d]^{(e_S,1)}&\\
&&*=0{}&
}\end{split}
\end{equation*}
It always has a unique predecessor $T_{N-1}$ of the form
\begin{equation*}
\begin{split}\xymatrix@R=10pt@C=12pt{
&&&&&\\
&&*=0{\circ}\ar@{-}[ul]\ar@{-}[ur]\ar@{}[u]|{S_1}
  \ar@{}[rr]|{\cdots}&&*=0{\circ}\ar@{-}[ul]\ar@{-}[ur]\ar@{}[u]|{S_n}
  \ar@{-}[ur]&\\
T_{N-1}=&&*=0{\bullet}\ar@{-}[u]^{(s_1,0)}\ar@{}[rr]|{\cdots}&&*=0{\bullet}\ar@{-}[u]_{(s_n,0)}&\\
&&&*=0{\circ}\ar@{-}[ur]_(.6){(s_n,1)}\ar@{-}[ul]^(.6){(s_1,1)}&&\\
&&&*=0{}\ar@{-}[u]_{(e_S,1)}&&
}\end{split}
\end{equation*}
where $S$ is of the form $(S_1,\cdots,S_n)\join[0]$.
\begin{equation*}
\begin{split}\xymatrix@R=10pt@C=12pt{
&&&&&\\
&&*=0{\circ}\ar@{-}[ul]\ar@{-}[ur]\ar@{}[u]|{S_1}
\ar@{}[rr]|{\cdots}&&*=0{\circ}\ar@{-}[ul]\ar@{-}[ur]\ar@{}[u]|{S_n}
\ar@{-}[ur]&\\
S=&&&*=0{\circ}\ar@{-}[ur]_{s_n}\ar@{-}[ul]^{s_1}&&\\
&&&*=0{}\ar@{-}[u]_{e_S}&&
}\end{split}
\end{equation*}
This observation already enables us to get
\end{rem}

\begin{proof}[Proof of parts (ii) and (iii) of Theorem \ref{thmsubcyl}]
Consider all the faces of $T_N$. For such a face $F\To T_N$, there are three
possibilities;
\begin{itemize}
\item[(a)] it misses an $S$-colour entirely (i.e. there is
an edge $s$ in $S$ so that neither $(s,0)$ nor $(s,1)$ are in $F$, so that
$\Omega[F]$ factors through $\bord\Omega[S]\otimes\Delta[1]$;
\item[(b)] $F$ is given by contracting the edge $(e_S,0)$,
in which case $\Omega[F]$ factors through $\Omega[T_{N-1}]$
(since the face $F$ then coincides with the face of $T_{N-1}$
obtained by contracting $(s_1,1),\ldots,(s_n,1)$);
\item[(c)] $F$ is given by chopping off the edge $(e_S,1)$
and the black vertex above it, i.e. $\Omega[F]=\Omega[S]\otimes\{0\}$.
This face cannot factor through $A_0$, nor through any of
the earlier percolation schemes since none of these has an edge
coloured $(e_S,0)$.
\end{itemize}
Thus, $\Omega[T_N]\cap A_{N-1}=\Lambda^r[T_N]$,
where $r$ denotes the black vertex as pictured above. This shows that
$$\xymatrix{
\Lambda^r[T]\ar[r]\ar[d]&A_{N-1}\ar[d]\\
\Omega[T]\ar[r]&A_N
}$$
is a pushout, exactly as stated in part (ii) of Theorem \ref{thmsubcyl}.
Moreover, the statement of part (iii) of Theorem \ref{thmsubcyl}
is obvious from the construction.
\end{proof}

\begin{paragr}
The proof of part (i) of Theorem \ref{thmsubcyl} is more involved, but
it is completely analogous to the proof of \cite[Proposition 9.2]{dend2}.
The difference with the situation in \emph{loc. cit.} is that, now, we are
dealing with an inclusion of the form
$$\bord\Omega[S]\otimes\Omega[T]\cup\Omega[S]\otimes\Lambda^e[T]\To\Omega[S]\otimes\Omega[T]\, ,$$
where $e$ is an outer edge of $T=i[1]$, whereas in \emph{loc. cit.}, we dealt
with
$$\Omega[S]\otimes\bord\Omega[T]\cup\Lambda^e[S]\otimes\Omega[T]\To\Omega[S]\otimes\Omega[T]\, ,$$
where $e$ is an inner edge of $S$. This forces us to look at different `spines' and
`characteristic edges' compared to the ones in \emph{loc. cit.} (notice also in this
connection that although the tensor product is symmetric, the partial order on the
percolation schemes is reversed).

The following lemma was also used (implicitly) in \cite{dend2}.
\end{paragr}

\begin{lemma}\label{thmsubcyl10}
Let $T_i$ and $T_j$ be two distinct percolation
schemes for $\Omega[S]\otimes\Delta[1]$. Then
$$\Omega[T_i]\cap\Omega[T_j]\subset\cup_k\Omega[T_k]$$
as subobjects of $\Omega[S]\otimes\Delta[1]$,
where the union ranges over all the percolation schemes $T_k$
which precede both $T_i$ and $T_j$ in the partial order.
\end{lemma}

\begin{proof}
Let $F$ be a common face of $\Omega[T_i]$ and $\Omega[T_j]$.
If $T_j\leq T_i$ in the partial order, there is nothing to prove.
Otherwise, we will give an algorithm for replacing $T_j$
by successively earlier percolation schemes,
$$T_j=T_{j_0}\geq T_{j_1}\geq T_{j_2}\geq\cdots$$
each having $F$ as a face, and eventually preceding $T_i$ in the partial
order. As a first step, $T_j$ is obtained from an earlier percolation scheme
$T_{j'}$ by changing
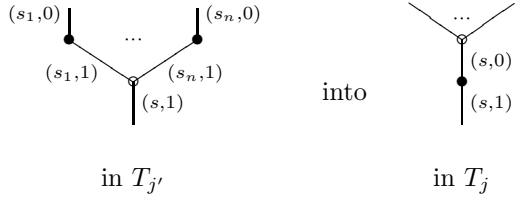
\begin{equation*}
\begin{split}\xymatrix@R=10pt@C=6pt{
&&&&\\
&*=0{\bullet}\ar@{}[rr]|{\cdots}\ar@{-}[u]^(.6){(s_1,0)}&&*=0{\bullet}\ar@{-}[u]_(.6){(s_n,0)}&\\
&&*=0{\circ}\ar@{-}[ur]_(.6){(s_n,1)}\ar@{-}[ul]^(.6){(s_1,1)}&&\\
&&*=0{}\ar@{-}[u]_{(s,1)}&&\\
&&\text{in $T_{j'}$}&&
}\end{split}\qquad \text{into}\quad
\begin{split}\xymatrix@R=10pt@C=6pt{
&&\\
&*=0{\circ}\ar@{-}[ur]\ar@{-}[ul]\ar@{}[u]|{\cdots}&\\
&*=0{\bullet}\ar@{-}[u]_{(s,0)}\ar@{-}[d]^{(s,1)}&\\
&*=0{}&\\
&\text{in $T_{j}$}&
}\end{split}
\end{equation*}
If $F$ is also a face of $T_{j'}$, we `push up the black vertices'
by replacing $T_j$ by $T_{j_1}=T_{j'}$. If not, then the colour $(s,0)$
must occur in $F$, hence in $T_j$
as well as in $T_i$. So the occurrence of $(s,0)$ in $T_j$
is not the reason that $T_j\nleqslant T_i$, and we put
$T_{j_1}=T_j$. Treating all black vertices in this way,
we can push them up if they occur below black vertices in $T_i$,
until we eventually reach a percolation scheme $T_{j_n}\leq T_j$,
still having $F$ as a face, for which $T_{j_n}\leq T_i$.
\end{proof}

\begin{paragr}
We return to the proof of Theorem \ref{thmsubcyl}.
Consider the inclusion
\begin{equation}\label{picspine0}
A_k\To A_{k+1}=A_k\cup\Omega[T_{k+1}]\, ,
\end{equation}
for $k+1<N$. The percolation scheme $T_{k+1}$
will have at least one black vertex. Consider
all the black vertices in $T_{k+1}$, and the corresponding
faces of $T_{k+1}$ which are formed by paths from these black
vertices to the root of $T_{k+1}$:
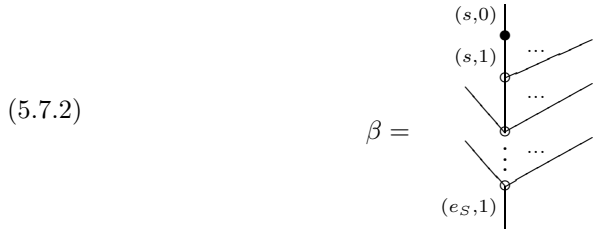
\begin{equation}\label{picspine}
\begin{split}\xymatrix@R=10pt@C=12pt{
&&&&\\
&&*=0{\bullet}\ar@{-}[u]^{(s,0)}&&\\
&&*=0{\circ}\ar@{-}[u]^{(s,1)}\ar@{}[ur]|(.65){\dots}\ar@{-}[urr]&&\\
\beta=&&*=0{\circ}\ar@{-}[u]\ar@{-}[lu]\ar@{}[ur]|(.65){\dots}\ar@{-}[urr]&&\\
&&*=0{\circ}\ar@{}[u]|(.65){\vdots}\ar@{-}[lu]\ar@{}[ur]|(.65){\dots}\ar@{-}[urr]&&\\
&&*=0{}\ar@{-}[u]^{(e_S,1)}&&
}\end{split}
\end{equation}
The face $\beta$ is the minimal external face which contains the given
black vertex as well as the root edge.
We call a face $\beta$ of $T_{k+1}$ of this form a \emph{spine} in
$T_{k+1}$. Notice that the vertex just above $(e_S,1)$ is indeed white, as in
the picture, because $k+1<N$. Notice also that the outer face of $\beta$ given
by chopping off this vertex misses the colour $e_S$, hence belongs to
$\bord\Omega[S]\otimes\Delta[1]\subset A_0$. Furthermore, the outer
face of $\beta$ given by chopping off its black top vertex belongs
to $\Omega[S]\otimes\{1\}\subset A_0$. Finally, all the inner faces of $\beta$
miss an $S$-color, hence factor through $\partial\Omega[S]\otimes\Delta[1]$,
except possibly the one given by contracting
the edge $(s,1)$ near the top.
However, if this last face $\partial_{(s,1)}(\beta)$ of $\beta$ belongs
to $A_k$, then some earlier $T_i$, $1\leq i\leq k$, contains the edge
$(s,0)$, hence all of $\beta$.
Thus, either $\Omega[\beta]$ is contained in
$A_k$, or we can adjoin it by an inner anodyne extension
\begin{equation}\label{picspine2}
\begin{split}
\xymatrix{
\Lambda^{(s,1)}[\beta]\ar[r]\ar[d]&A_k\ar[d]\\
\Omega[\beta]\ar[r]&A_k\cup\Omega[\beta]\, .
}\end{split}
\end{equation}
Such a spine $\beta$ is an example of an \emph{initial segment} of
$T_{k+1}$. Recall from \cite{dend2} that a face $R\To T_{k+1}$
is called an \emph{initial segment}
if it is obtained by successively chopping off top vertices. Our
strategy will be to adjoin more initial segments of $T_{k+1}$ to $A_k$,
starting with the spines. To this end, we need the following definition
and lemma from \cite{dend2}, in which we use the notation $m(R)\subset\Omega[T_{k+1}]$
for the image of the map $\Omega[R]\To\Omega[T_{k+1}]$ given by an initial
segment $R$.
\end{paragr}

\begin{definition}[{\cite{dend2}}]\label{thmsubcyl11}
Let $R$, $Q_1,\ldots,Q_p$ be initial segments of $T_{k+1}$, and
let $B=m(Q_1)\cup\cdots\cup m(Q_p)$. Suppose that, for every top
face $F$ of $R$, we have $m(F)\subset A_k\cup B$.
In this situation, an inner edge $\xi$ of $R$ is called \emph{characteristic}
with respect to $Q_1,\ldots,Q_p$ if, for any inner face $F$ of $R$,
if $m(F/\xi)$ is contained in $A_k\cup B$, then so is $m(F)$
(where $F/\xi\To F$ is the face obtained by contracting $\xi$).
\end{definition}

\begin{example}
In any spine $\beta$ as in picture \eqref{picspine}, the edge $\xi=(s,1)$
is characteristic with respect to any family of initial segments.
\end{example}

\begin{example}\label{thmsubcyl12}
More generally, suppose $R$ is an initial segment of
$T_{k+1}$ given by a spine $\beta$ expanded by one (or more)
white vertices, say
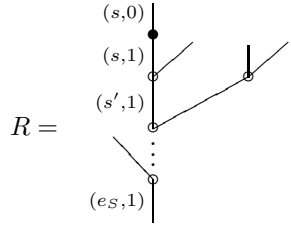
\begin{equation*}
\begin{split}\xymatrix@R=10pt@C=12pt{
&&&&&\\
&&*=0{\bullet}\ar@{-}[u]^{(s,0)}&&&\\
&&*=0{\circ}\ar@{-}[u]^{(s,1)}\ar@{-}[ur]&&*=0{\circ}\ar@{-}[u]\ar@{-}[ur]&\\
R=&&*=0{\circ}\ar@{-}[u]^{(s',1)}\ar@{-}[urr]&&&\\
&&*=0{\circ}\ar@{}[u]|(.65){\vdots}\ar@{-}[lu]&&&\\
&&*=0{}\ar@{-}[u]^{(e_S,1)}&&&
}\end{split}
\end{equation*}
Then $\xi=(s,1)$ is again characteristic with respect to any
family $Q_1,\ldots,Q_p$. Indeed, if $R/\xi$ is a face
of an initial segment $Q_i$, then so is $R$ itself; see \cite[Remark 9.6~(iv)]{dend2}.
And if $R/\xi$ is a face of $T_j$ for a percolation scheme $T_j$, then
$T_j$ either contains $R$, or looks like
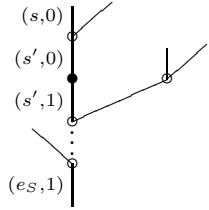
\begin{equation*}
\begin{split}\xymatrix@R=10pt@C=12pt{
&&&&&\\
&&*=0{\circ}\ar@{-}[u]^{(s,0)}\ar@{-}[ur]&&&\\
&&*=0{\bullet}\ar@{-}[u]^{(s',0)}&&*=0{\circ}\ar@{-}[u]\ar@{-}[ur]&\\
&&*=0{\circ}\ar@{-}[u]^{(s',1)}\ar@{-}[urr]&&&\\
&&*=0{\circ}\ar@{}[u]|(.68){\vdots}\ar@{-}[lu]&&&\\
&&*=0{}\ar@{-}[u]^{(e_S,1)}&&&
}\end{split}
\end{equation*}
But, by Lemma \ref{thmsubcyl10}, we can assume $T_j$ comes
before $T_{k+1}$ in the partial order, so this is impossible.
Finally, if $\Omega[R/\xi]\To\Omega[S]\otimes\Delta[1]$
factors through $A_0$, then $R/\xi$ misses an $S$-colour, and hence so
does $R$.
\end{example}

\begin{paragr}
The proof of Theorem \ref{thmsubcyl}~(i) is based on a repeated
use of arguments like the preceding one in Example \ref{thmsubcyl12}.
We quote the following lemma on characteristic edges from \cite{dend2}.
\end{paragr}

\begin{lemma}[{\cite[Lemma 9.7]{dend2}}]\label{thmsubcyl13}
Let $R$, $Q_1,\ldots,Q_p$ be initial segments of $T_{k+1}$. Let
$B=m(Q_1)\cup\cdots\cup m(Q_p)$, and suppose each top face of
$R$ has the property that $m(F)$ is contained in $A_k\cup B$.
If $R$ possesses a characteristic edge with respect to $Q_1,\ldots,Q_p$,
then the inclusion
$$A_k\cup B \To A_k\cup B\cup m(R)$$
of subobjects of $\Omega[S]\otimes\Delta[1]$ is inner anodyne.
\end{lemma}

\begin{lemma}\label{thmsubcyl14}
Let $R$, $Q_1,\ldots,Q_p$ be initial segments of $T_{k+1}$, satisfying
condition (i) in Definition \ref{thmsubcyl11}, and let $\beta$
be a spine in $R$. Then the edge $\xi=(s,1)$ immediately
below the black vertex on the spine is a characteristic edge for $R$.
\end{lemma}

\begin{proof}[Hint for a proof]
This is proved exactly as Example \ref{thmsubcyl12}; cf. also
\cite[Lemma 9.8]{dend2}.
\end{proof}

\begin{paragr}
Using the characteristic edges from Lemma \ref{thmsubcyl14}, one can now
copy the proof of \cite[Lemma 9.9]{dend2}, repeated below as Lemma \ref{thmsubcyl15},
verbatim. This proof is by induction on $l$, and describes a precise
strategy for adjoining more and/or larger initial segments of $T_{k+1}$ to $A_k$.
\end{paragr}

\begin{lemma}\label{thmsubcyl15}
Fix $l\leq 0$, and let $Q_1,\ldots,Q_p$ be a family of initial
segments in $T_{k+1}$, each containing at least one spine, and at most
$l$ spines (so, necessarily, $p=0$ if $l=0$). Let
$R_1,\ldots,R_q$ be initial segments which each contain exactly $l+1$
spines. Then the inclusion $A_k\To A_k\cup B\cup C$ is inner anodyne,
where $B=m(Q_1)\cup\cdots\cup m(Q_p)$ and
$C=m(R_1)\cup\cdots\cup m(R_q)$.
\end{lemma}

\begin{paragr}
This strategy terminates when one arrives at the number $l$ of \emph{all}
spines in $T_{k+1}$. Indeed, for this $l$ and $p=0$, $q=1$, Lemma
\ref{thmsubcyl15} states for $R_1=T_{k+1}$ that $A_k\To A_{k+1}$
is inner anodyne, as asserted in Theorem \ref{thmsubcyl}~(i). This
completes the proof of Theorem \ref{thmsubcyl}.
\end{paragr}

\section{$\infty$-operads as fibrant objects}\label{section6}

\begin{paragr}\label{dendanalogofk}
The aim of this section is to characterize $\infty$-operads
as the fibrant objects of the model category structure on the
the category of dendroidal sets given by Proposition \ref{cmfdendbasic}.
This characterization is stated in Theorem \ref{characfibinftyopermain}
below.

Given an $\infty$-category $X$, we denote by $k(X)$ the maximal
Kan complex contained in $X$; see \cite[Corollary 1.5]{joyal}.

Recall that, given two dendroidal sets $A$ and $X$, we write
$$\map(A,X)=i^*\sHom(A,X)\, .$$
Note that, by virtue of Proposition
\ref{tensinneranod}, if $X$ is an $\infty$-operad, and if $A$
is normal, then $\sHom(A,X)$ is an $\infty$-operad, so that $\map(A,X)$ is
an $\infty$-category.

For an $\infty$-operad $X$ and a simplicial set $K$, we will write
$X^{(K)}$ for the subcomplex of $\sHom(i_!(K),X)$ which consists
of dendrices
$$a:\Omega[T]\times i_!(K)\To X$$
such that, for any $0$-cell $u$ in $T$, the induced map
$$a_u:K\To i^*(X)$$
factors through $k(i^*(X))$ (i.e. all the $1$-cells in the image
of $a_u$ are weakly invertible in $i^*(X)$).

For an $\infty$-operad $X$ and a normal dendroidal set $A$, we will
write $k(A,X)$ for the subcomplex of $\map(A,X)$ which consists
of maps
$$u: A\otimes i_!(\Delta[n])\To X$$
such that, for all vertices $a$ of $A$ (i.e. maps $a:\eta\To A$), the induced map
$$u_a:\Delta[n]\To i^*(X)$$
factors through $k(i^*(X))$.
So, by definition, for any normal dendroidal set $A$, any simplicial set $K$,
and any $\infty$-operad $X$, there is a natural bijection:
\begin{equation}\label{dendanalogofk1}
\Hom_{\sset}(K,k(A,X))\simeq\Hom_{\dset}(A,X^{(K)})\, .
\end{equation}
\end{paragr}

\begin{rem}\label{remtermwiseweakequiv}
The simplicial set $k(A,X)$ is by definition the
$\infty$-category of objectwise weakly invertible $1$-cells
in $\map(A,X)$. We can reformulate the definition
of $k(A,X)$ as follows (still with $A$ normal and $X$ an
$\infty$-operad). Define
\begin{equation}
\mathrm{Ob}\, A=\coprod_{A_0}\eta\, .
\end{equation}
We have a unique monomorphism $i:\mathrm{Ob}\, A\To A$
which is the identity on $0$-cells. As $A$ is normal,
$i$ is a normal monomorphism. We also have
\begin{equation}
k(\mathrm{Ob}\, A,X)=k(\map(\mathrm{Ob}\, A,X))=\prod_{A_0}k(i^*X)\, ,
\end{equation}
and $k(A,X)$ fits by definition in the following pullback square.
\begin{equation}\begin{split}\xymatrix{
k(A,X)\ar[r]\ar[d]&\map(A,X)\ar[d]\\
k(\map(\mathrm{Ob}\, A,X))\ar[r]&\map(\mathrm{Ob}\, A,X)
}\end{split}\end{equation}
In particular, the projection of $k(A,X)$ on
$k(\mathrm{Ob}\, A,X)$ is an inner Kan fibration, and
as the latter is a Kan complex, this shows that $k(A,X)$
is an $\infty$-category. One of the key results
of this section asserts that $k(A,X)$ is a Kan complex
as well, which can be reformulated by saying that
the inclusion $k(\map(A,X))\subset k(A,X)$ is in fact an
equality. In other words, a map in the $\infty$-category
$\map(A,X)$ is weakly invertible if and only it is objectwise weakly invertible;
see Corollary \ref{dendanalogofk3etdemi}.
\end{rem}

\begin{paragr}\label{taucatfibraequivev1fibtriv}
Before stating the next theorem, we recall
that, for a morphism between $\infty$-categories
$f:X\To Y$, the induced map $\tau(f):\tau(X)\To\tau(Y)$ is
a categorical fibration if and only if the map
$$\mathit{ev}_1:X^{(\Delta[1])}\To Y^{(\Delta[1])}\times_Y X$$
induced by evaluating at $1$ (i.e. by the inclusion $\{1\}\To\Delta[1]$)
has the right lifting property with respect to $\bord\Delta[0]\To\Delta[0]$;
see \cite[Proposition 2.4]{joyal}.
\end{paragr}

\begin{thm}\label{dendanalogofk2}
Let $p:X\To Y$ be an inner Kan fibration between $\infty$-operads.
The map $\mathit{ev}_1:X^{(\Delta[1])}\To Y^{(\Delta[1])}\times_Y X$
has the right lifting property with respect to inclusions
$\bord\Omega[S]\To\Omega[S]$ for any tree $S$ with at least
one vertex. Consequently, the functor $\tau i^*(p)$ is a
categorical fibration if and only if
the evaluation at $1$ map $X^{(\Delta[1])}\To Y^{(\Delta[1])}\times_Y X$ is
a trivial fibration of dendroidal sets.
\end{thm}

\begin{proof}
Consider a tree $S$ with at least one vertex and a solid
commutative square
$$\xymatrix{
\bord\Omega[S]\ar[r]^f\ar[d]&X^{(\Delta[1])}\ar[d]\\
\Omega[S]\ar[r]^(.4)g\ar@{..>}[ur]^h&Y^{(\Delta[1])}\times_Y X
}$$
We want to prove the existence of a diagonal filling $h$.
This corresponds by adjunction to a filling $\tilde h$ in the
following commutative square
$$\xymatrix{
\bord\Omega[S]\otimes\Delta[1]\cup\Omega[S]\otimes\{1\}\ar[r]^(.75){\tilde f}\ar[d]&X\ar[d]\\
\Omega[S]\otimes\Delta[1]\ar[r]^(.6){\tilde g}\ar@{..>}[ur]^{\tilde h}&Y
}$$
(as the inclusion of $\bord\Omega[S]$ in $\Omega[S]$ is
bijective on objects, and as the restriction of $\tilde h$
to $\bord\Omega[S]\otimes\Delta[1]\cup\Omega[S]\otimes\{1\}$
coincides with $\tilde f$, the map $\Omega[S]\To X^{\Delta[1]}$
corresponding to a filling $\tilde h$ will automatically
factor through $X^{(\Delta[1])}$).

Consider the filtration
$$\bord\Omega[S]\otimes\Delta[1]\cup\Omega[S]\otimes\{1\}=
A_0\subset A_1\subset\cdots\subset A_{N-1}\subset A_N=\Omega[S]\otimes\Delta[1]$$
given by Theorem \ref{thmsubcyl}. As the map $X\To Y$ is an inner Kan fibration,
using Theorem \ref{thmsubcyl}~(i), it is sufficient to find a filling
in a solid commutative diagram of shape
$$\xymatrix{
A_{N-1}\ar[r]^{f'}\ar[d]&X\ar[d]\\
\Omega[S]\otimes\Delta[1]\ar[r]^(.6){\tilde g}\ar@{..>}[ur]^{\tilde h}&Y
}$$
in which the restriction of $f'$ to $\bord\Omega[S]\otimes\Delta[1]\cup\Omega[S]\otimes\{1\}$
coincides with $\tilde f$. By virtue of Theorem \ref{thmsubcyl}~(ii),
it is even sufficient to find a filling $k$ in a solid commutative diagram
of shape
$$\xymatrix{
\Lambda^r[T]\ar[r]^{a}\ar[d]&X\ar[d]\\
\Omega[T]\ar[r]^(.6){b}\ar@{..>}[ur]^{k}&Y
}$$
in which $T$ is a tree with unary vertex $r$ at the root, and 
$a$ is the restriction of $f'$ to $\Lambda^r[T]\subset A_{N-1}$.
Furthermore, by Theorem \ref{thmsubcyl}~(iii), we may assume
that $a(r)$ is weakly invertible in $i^*(X)$.
Thus, the existence of the filling $k$
is ensured by Theorem \ref{dendjoyal}.

The last assertion of the theorem follows from \ref{taucatfibraequivev1fibtriv}.
\end{proof}

\begin{lemma}\label{predendanalogofk3}
Any left fibration between Kan complexes is
a Kan fibration.
\end{lemma}

\begin{proof}
This follows from \cite[Theorem 2.2 and Proposition 2.7]{joyal}.
\end{proof}

\begin{lemma}\label{predendanalogofk4}
A morphism of simplicial sets $X\To Y$ is a left (resp. right)
fibration if and only it has the right lifting property with
respect to maps of shape
$$\bord\Delta[n]\times\Delta[1]\cup\Delta[n]\times\{e\}\To\Delta[n]\times\Delta[1]$$
for $e=1$ (resp. for $e=0$) and $n\geq 0$.
\end{lemma}

\begin{proof}
The map $\bord\Delta[n]\times\Delta[1]\cup\Delta[n]\times\{0\}\To\Delta[n]\times\Delta[1]$
is obtained as a finite composition of pushouts of horns of shape
$\Lambda^k[n+1]\To\Delta[n+1]$ with $0\leq k<n+1$;
see (the dual version of) \cite[Chapter IV, 2.1.1]{GZ}.

Conversely, the inclusion map $\Lambda^k[n]\To\Delta[n]$, $0\leq k< n$, is a retract
of the map $\Lambda^k[n]\times\Delta[1]\cup\Delta[n]\times\{0\}\To\Delta[n]\times\Delta[1]$;
see \cite[Chapter IV, 2.1.3]{GZ}.

We deduce easily from this that a morphism of simplicial sets $X\To Y$
is a right fibration
if and only if the evaluation at $0$ map $X^{\Delta[1]}\To Y^{\Delta[1]}\times_Y X$
is a trivial fibration (i.e. has the right lifting property with respect
to monomorphisms). The case of left fibrations follows by duality.
\end{proof}

\begin{prop}\label{dendanalogofk3}
Let $p:X\To Y$ be an inner Kan fibration
between $\infty$-operads. If $\tau i^*(p)$ is
a categorical fibration, then, for any
monomorphism between normal dendroidal sets
$A\To B$, the map
$$k(B,X)\To k(B,Y)\times_{k(A,Y)}k(A,X)$$
is a Kan fibration between Kan complexes.
\end{prop}

\begin{proof}
The functor $i_!:\sset\to\dset$ being
symmetric monoidal and preserving inner anodyne
extensions, Proposition \ref{tensinneranod} implies
that the map
$$\map(B,X)\To \map(B,Y)\times_{\map(A,Y)}\map(A,X)$$
is an inner Kan fibration between $\infty$-categories.
This implies the map
$$k(B,X)\To k(B,Y)\times_{k(A,Y)}k(A,X)$$
is an inner Kan fibration between $\infty$-categories.

We claim that this map has the right lifting property with respect to
the inclusion $\{1\}\To\Delta[1]$.
Using the identification \eqref{dendanalogofk1}, we see that
lifting problems of shape
\begin{equation}\label{dendanalogofk3000}\begin{split}
\xymatrix{
\{1\}\ar[r]\ar[d]&k(B,X)\ar[d]\\
\Delta[1]\ar[r]\ar@{..>}[ur]&k(B,Y)\times_{k(A,Y)}k(A,X)
}\end{split}\end{equation}
correspond to lifting problems of shape
\begin{equation}\label{dendanalogofk30000}\begin{split}
\xymatrix{
A\ar[r]\ar[d]& X^{(\Delta[1])}\ar[d]\\
B\ar[r]\ar@{..>}[ur]&Y^{(\Delta[1])}\times_Y X
}\end{split}\end{equation}
so that our claim follows from Theorem \ref{dendanalogofk2}.

More generally, the map $k(B,X)\To k(B,Y)\times_{k(A,Y)}k(A,X)$
has the right lifting property with respect to maps of shape
\begin{equation}\label{dendanalogofk30}\begin{split}
\bord\Delta[n]\times\Delta[1]\cup\Delta[n]\times\{1\}\To\Delta[n]\times\Delta[1]\ ,
\quad n\geq 0\, .
\end{split}\end{equation}

We have just checked it above in the case where $n=0$, so that it remains
to prove the case where $n>0$. Consider a lifting problem of shape
\begin{equation}\label{dendanalogofk31}\begin{split}
\xymatrix{
\bord\Delta[n]\times\Delta[1]\cup\Delta[n]\times\{1\}\ar[r]^(.67)u\ar[d]&k(B,X)\ar[d]\\
\Delta[n]\times\Delta[1]\ar[r]_v\ar@{..>}[ur]_g&k(B,Y)\times_{k(A,Y)}k(A,X)
}\end{split}\end{equation}
This lifting problem gives rise to a lifting problem of shape
\begin{equation}\label{dendanalogofk32}\begin{split}
\xymatrix{
\bord\Omega[n]\otimes B\cup \Omega[n]\otimes A\ar[r]\ar[d]& X^{(\Delta[1])}\ar[d]\\
\Omega[n]\otimes B\ar[r]\ar@{..>}[ur]_h& Y^{(\Delta[1])}\times_Y X\, .
}\end{split}\end{equation}
The existence of the lifting $h$ is provided again by
Theorem \ref{dendanalogofk2}. The map $h$ defines a map
$$l:i_!(\Delta[n]\times\Delta[1])\otimes B\To X\, .$$
As a consequence,
it is sufficient to check that, for every non-degenerate $m$-simplex
$\delta:\Delta[m]\To\Delta[n]\times\Delta[1]$, $m\geq 1$, and for any object
$b:\eta\To B$, the map $(l i_!(\delta))_b:\Delta[m]\To i^*(X)$
factors through $k(i^*(X))$. Using the `$2$ out of $3$ property'
for weakly invertible $1$-cells in $i^*(X)$, we can assume that $m=1$.
But then, as $n>0$, using again the `$2$ out of $3$ property' for
weakly invertible $1$-cells,
we may assume that $\delta$ factors through $\bord\Delta[n]\times\Delta[1]$,
which implies then that $(l\circ i_!(\delta))\otimes 1_B$ factors through the subcomplex
$i_!(\bord\Delta[n]\times\Delta[1])\otimes B$: 
the required property thus follows from the fact that the restriction
of the transpose of $h$ to the object
$i_!(\Delta[n]\times\Delta[1]\cup\bord\Delta[n]\times\{1\})\otimes B$
corresponds to the map $u$ in \eqref{dendanalogofk31}.

By virtue of Lemma \ref{predendanalogofk4}, the map
$k(B,X)\To k(B,Y)\times_{k(A,Y)}k(A,X)$ is a left fibration, hence,
by \cite[Proposition 2.7]{joyal}, is conservative.
By applying \cite[Corollary 1.4]{joyal}, we deduce, from
the case where $A=\varnothing$ and $Y$ is the terminal dendroidal set,
that $k(B,X)$ is a Kan
complex for any normal dendroidal set $B$ and any $\infty$-operad $X$. As
any left fibration between Kan complexes is a Kan fibration
(Lemma \ref{predendanalogofk3}),
the maps $k(B,X)\To k(A,X)$ are thus Kan fibrations between Kan
complexes for any monomorphisms between normal dendroidal sets
$A\To B$ and any $\infty$-operad $X$. As a consequence, Kan
fibrations being stable by pullback, we see that the fiber product
$k(B,Y)\times_{k(A,Y)}k(A,X)$ is a Kan complex.
Using again Lemma \ref{predendanalogofk3}, we conclude that
$k(B,X)\To k(B,Y)\times_{k(A,Y)}k(A,X)$ is a Kan fibration
between Kan complexes.
\end{proof}

\begin{cor}\label{dendanalogofk3etdemi}
For any normal dendroidal set $A$ and any $\infty$-operad $X$,
we have
$$k(\map(A,X))=k(A,X)\, .$$
For any inner Kan fibration
between $\infty$-operads $p:X\To Y$
such that $\tau i^*(p)$ is a categorical fibration,
and for any monomorphism between normal dendroidal sets
$A\To B$, we have
$$k(\map(B,Y)\times_{\map(A,Y)}\map(A,X))=k(B,Y)\times_{k(A,Y)}k(A,X) \, .$$
\end{cor}

\begin{proof}
If $A$ is normal, then, for any operad $X$, $k(A,X)$
is a Kan complex which contains $k(\map(A,X))$.
As $k(\map(A,X))$ is the maximal sub Kan complex
contained in the $\infty$-category $\map(A,X)$, this
proves the first assertion. The second
assertion is proved similarly.
\end{proof}

\begin{cor}\label{dendanalogofk4}
Let $p:X\To Y$ be an inner Kan fibration
between $\infty$-operads. If $\tau i^*(p)$
is a categorical fibration,
then, for any anodyne extension of simplicial sets $K\To L$,
the map
$$X^{(L)}\To Y^{(L)}\times_{Y^{(K)}}X^{(K)}$$
is a trivial fibration of dendroidal sets.
\end{cor}

\begin{proof}
This follows from Proposition \ref{dendanalogofk3}
and from the natural identification \eqref{dendanalogofk1}.
\end{proof}

\begin{thm}\label{characfibinftyopermain}
A dendroidal set is $J$-fibrant if and only if it is
an $\infty$-operad. An inner Kan fibration between
$\infty$-operads $p:X\To Y$ is a $J$-fibration
(i.e. a fibration for the model category structure
of Proposition \ref{cmfdendbasic}) if and only if
$\tau i^*(p)$ is a categorical fibration.
\end{thm}

\begin{proof}
Let $p:X\To Y$ be inner Kan fibration between
$\infty$-operads. We have to prove that,
for $e=0,1$, the anodyne extension $\{e\}\To J$ induces
a trivial fibration of
dendroidal sets
$$X^{J_d}\To Y^{J_d}\times_Y X$$
if and only if $\tau i^*(p)$ is a categorical fibration.
But, for any $\infty$-operad $Z$, we clearly have
$Z^{J_d}=Z^{(J)}$ and $Z=Z^{(\{e\})}$. Hence, by virtue of
Corollary \ref{dendanalogofk4}, if $\tau i^*(p)$ is a categorical fibration,
then $p$ is a $J$-fibration. The converse is a direct consequence
of \cite[Corollary 1.6]{joyal}.
\end{proof}

\begin{cor}\label{charoperweakequiv}
The class of weak operadic equivalences is the smallest
class of maps of dendroidal sets $\Wpr$ which satisfies the
following three properties.
\begin{itemize}
\item[(a)] (`$2$ out $3$ property')
In any commutative triangle, if two maps are in $\Wpr$, then
so is the third.
\item[(b)] Any inner anodyne extension is in $\Wpr$.
\item[(c)] Any trivial fibration between $\infty$-operads
is in $\Wpr$.
\end{itemize}
\end{cor}

\begin{proof}
Consider a class of maps $\Wpr$ satisfying conditions
(a), (b) and (c) above. We want to prove that any
weak operadic equivalence is in $\Wpr$.

Let $f:A\To B$ be a morphism of dendroidal sets.
Using the small object argument applied to the set of
inner horns, we can see there exists a commutative
square
$$\xymatrix{
A\ar[r]^a\ar[d]_f&X\ar[d]^p\\
B\ar[r]^b&Y
}$$
in which the maps $a$ and $b$ are inner anodyne extensions,
and $X$ and $Y$ are $\infty$-operads.
It is clear that $f$ is a weak operadic equivalence
(resp. is in $\Wpr$) if and only $p$ has the
same property. Hence it is sufficient to prove that
any weak operadic equivalence between $\infty$-operads
is in $\Wpr$. As any trivial fibration
between $\infty$-operads is in $\Wpr$ by assumption,
and as $\infty$-operads are the fibrant objects
of a model category, this corollary follows from Ken~Brown's
Lemma \cite[Lemma 1.1.12]{Ho}.
\end{proof}

\begin{paragr}
We will write $C_n$ for the corolla with $n+1$ edges,
\begin{equation*}
\begin{split}
\xymatrix@R=10pt@C=12pt{
&&&&&&&\\
&*=0{}&&*=0{}\ar@{}[rrr]|{\cdots\cdots}&&&*=0{}&&\\
C_n=&&&*=0{\bullet}\ar@{-}[u]_{a_2}\ar@{-}[ull]^{a_1}\ar@{-}[urrr]_{a_n}&&&&.\\
&&&*=0{}\ar@{-}[u]^a&&&&
}
\end{split}\end{equation*}

Let $X$ be an $\infty$-operad.
Given an $(n+1)$-tuple of $0$-cells $(x_1,\ldots,x_n,x)$ in $X$,
the space of maps $X(x_1,\ldots,x_n;x)$ is obtained
by the pullback below, in which the map $p$ is
the map induced by the inclusion $\eta\amalg\cdots\amalg\eta\To\Omega[C_n]$
(with $n+1$ copies of $\eta$, corresponding to the $n+1$ objects
$(a_1,\ldots,a_n,a)$ of $C_n$).
$$\xymatrix{
X(x_1,\ldots,x_n;x)\ar[r]\ar[d]&\sHom(\Omega[C_n],X)\ar[d]^p\\
\eta\ar[r]_(.45){(x_1,\ldots,x_n,x)}&X^{n+1}
}$$
Using the identification $\sset=\dset/\eta$, we shall consider
$X(x_1,\ldots,x_n;x)$ as a simplicial set.
\end{paragr}

\begin{prop}
The simplicial set $X(x_1,\ldots,x_n;x)$
is a Kan complex.
\end{prop}

\begin{proof}
The first assertion of Corollary \ref{dendanalogofk3etdemi}
for $A=\Omega[C_n]$ can be reinterpreted by saying
we have the pullback square below (see Remark \ref{remtermwiseweakequiv}).
$$\xymatrix{
k(\map(\Omega[C_n],X))\ar[r]\ar[d]&\sHom(\Omega[C_n],X)\ar[d]\\
k(i^*X)^{n+1}\ar[r]&X^{n+1}
}$$
As the terminal simplicial set $\eta$ is certainly a Kan complex,
it thus follows from the construction of $X(x_1,\ldots,x_n;x)$
that we have a pullback square
$$\xymatrix{
X(x_1,\ldots,x_n;x)\ar[r]\ar[d]&k(\map(\Omega[C_n],X))\ar[d]\\
\eta\ar[r]_(.3){(x_1,\ldots,x_n,x)}&k(i^*X)^{n+1}
}$$
in which the right vertical map in this diagram is a Kan
fibration (by Proposition \ref{dendanalogofk3},
applied for $A=\eta\amalg\cdots\amalg\eta$ and
$B=\Omega[C_n]$). The stability of Kan
fibrations by pullback achieves the proof.
\end{proof}

\begin{prop}\label{pi0Hominftyoperads}
There is a canonical bijection
$$\pi_0(X(x_1,\ldots,x_n;x))\simeq
\tau_d(X)(x_1,\ldots,x_n;x)\, .$$
\end{prop}

\begin{proof}
We will use the explicit description
of $\tau_d(X)$ given by \cite[Lemma 6.4 and Proposition 6.6]{dend2}.
The unit map $X\To\nerf_d\tau_d(X)$ induces a map
$$X(x_1,\ldots,x_n;x)\To
(\nerf_d\tau_d(X))(x_1,\ldots,x_n;x)\, .$$
It is easily seen that $(\nerf_d\tau_d(X))(x_1,\ldots,x_n;x)$
is the discrete simplicial set associated to
$\tau_d(X)(x_1,\ldots,x_n;x)$, so that we get a surjective map
$$\pi_0(X(x_1,\ldots,x_n;x))\To
\tau_d(X)(x_1,\ldots,x_n;x)\, .$$
Using the explicit description
of $\tau_d(X)$ given by \cite[Lemma 6.4 and Proposition 6.6]{dend2},
it is now sufficient to prove that, if $f$ and $g$
are two $0$-simplices of $X(x_1,\ldots,x_n;x)$
which are homotopic along the edge $0$
in the sense of \cite[Definition 6.2]{dend2}, then they
belong to the same connected component.
But then, $f$ and $g$ define two objects of
$\tau(\map(\Omega[C_n],X))$ which are isomorphic,
which can be expressed by the existence of a map
$$h:\Delta[1]\To k(\map(\Omega[C_n],X))$$
which connect $f$ and $g$. Using that
$k(\map(\Omega[C_n],X))\To i^*(X)^{n+1}$ is a Kan fibration
between Kan complexes,
we can see by a path lifting argument that such a map $h$
is homotopic under $\partial\Delta[1]$ to
a map $\Delta[1]\To X(x_1,\ldots,x_n;x)$ which connects $f$
and $g$.
\end{proof}

\begin{lemma}\label{trivfibinftyoperadsHom}
Let $X\To Y$ be a trivial fibration between $\infty$-operads.
Then, for any $(n+1)$-tuple of $0$-cells $(x_1,\ldots,x_n,x)$ in $X$,
the induced map
$$X(x_1,\ldots,x_n;x)\To Y(f(x_1),\ldots,f(x_n);f(x))$$
is a trivial fibration of simplicial sets.
\end{lemma}

\begin{proof}
We know that the map
$$\sHom(\Omega[C_n],X)\To\sHom(\Omega[C_n],Y)
\times_{Y^{n+1}}X^{n+1}$$
is a trivial fibration (this follows from Proposition \ref{BVtensOK}
by adjunction).
As we have a pullback of shape
$$\xymatrix{
X(x_1,\ldots,x_n;x)\ar[r]\ar[d]&\sHom(\Omega[C_n],X)\ar[d]\\
Y(f(x_1),\ldots,f(x_n);f(x))\ar[r]&
\sHom(\Omega[C_n],Y)
\times_{Y^{n+1}}X^{n+1}
}$$
this proves the lemma.
\end{proof}

\begin{prop}\label{realdendoper}
The functor $\tau_d:\dset\To\oper$ sends weak
operadic equivalences to equivalences of operads.
\end{prop}

\begin{proof}
We know that $\tau_d$ sends inner horn inclusions
to isomorphisms of operads (this follows
from \cite[Theorem 6.1]{dend2} by the Yoneda Lemma). As $\tau_d$
preserves colimits, we deduce that $\tau_d$
sends inner anodyne extensions to isomorphisms
of operads. By virtue of Corollary \ref{charoperweakequiv},
it is thus sufficient to prove that $\tau_d$ sends
trivial fibrations between $\infty$-operads to
equivalences of operads. Let $f:X\To Y$ be a trivial fibration
between $\infty$-operads. By virtue of Proposition \ref{pi0Hominftyoperads}
and of Lemma \ref{trivfibinftyoperadsHom}, we see that $\tau_d(f)$ is fully faithful.
As $f$ is obviously surjective on
$0$-cells, $\tau_d(f)$ has to be an equivalence of operads.
\end{proof}

\begin{cor}\label{quillendendoper}
The adjunction $\tau_d:\dset\rightleftarrows\oper:\nerf_d$
is a Quillen pair. Moreover, the two functors $\tau_d$
and $\nerf_d$ both preserve weak equivalences.
In particular, a morphism of operads is an equivalence of
operads if and only if its dendroidal nerve is
a weak operadic equivalence.
\end{cor}

\begin{proof}
The functor $\tau_d$ preserves cofibrations, so that
this is a direct consequence of Proposition \ref{realdendoper}.
Note that any operad is fibrant, so that the dendroidal nerve
functor $\nerf_d$ preserves weak equivalences. Hence
the last assertion comes from the fact $\nerf_d$ is fully faithful
and $\tau_d$ preserves weak equivalences.
\end{proof}

\begin{rem}
Theorem \ref{characfibinftyopermain} also asserts that
the functor $\tau_d$ preserves fibrations between $\infty$-operads.
\end{rem}

\begin{paragr}
If $A$ is a normal dendroidal set, and if $X$ is an $\infty$-operad,
we have
\begin{equation}\label{htpyclassesmaps0}
\Hom_{\ho(\dset)}(A,X)=[A,X]\simeq \pi_0(k(\map(A,X)))\, .
\end{equation}
Indeed, $J\otimes A$ is a cylinder of $A$, and morphisms
$$J\otimes A\To X$$
correspond to morphisms
$$J\To k(\map(A,X))\, ,$$
so that this formula follows from the fact $X$ is $J$-fibrant.
The next statement is a reformulation of \eqref{htpyclassesmaps0}.
\end{paragr}

\begin{prop}\label{htpyclassesmaps}
Let $A$ be a normal dendroidal set, and $X$ an $\infty$-operad.
The set $[A,X]=\Hom_{\ho(\dset)}(A,X)$ can be canonically identified with the
set of isomorphism classes of objects in the category
$\tau\map(A,X)$ (which is also the category underlying
$\tau_d(\sHom(A,X))$).
\end{prop}

\begin{proof}
This proposition is a direct application
of the explicit description of the operad $\tau_d(\sHom(A,X))$
given by \cite[Proposition 6.6]{dend2} and of Corollary \ref{dendanalogofk4}.
\end{proof}

\begin{appendices}
\section{Grafting orders onto trees}\label{appA}

The main goal of the technical sections \ref{section4} and \ref{section5}
was to deduce Theorem \ref{characfibinftyopermain}, and from it,
Corollary \ref{dendanalogofk4}. There is an asymmetry in this approach, in that
Theorem \ref{characfibinftyopermain} was only proved for evaluation at one of the
end points, and the symmetry was established in Corollary \ref{dendanalogofk4}
by using the theory of left fibrations between simplicial sets.

In these two appendices, we will prove the analogs of
Theorems \ref{dendjoyal} and \ref{thmsubcyl}, from which one can deduce
directly the symmetric version of Theorem \ref{characfibinftyopermain}
(for evaluation at $0$). These two appendices can also be used as an alternative
approach to the results in Section \ref{section6}.
Moreover, they are of interest by themselves, as they form the basis of a theory
of right fibrations of dendroidal sets.

However, since the left-right duality for simplicial sets does not extend to
dendroidal sets, the results of these appendices cannot be deduced from their
analogs proved earlier.

We begin by studying the analog of
Theorem \ref{dendjoyal} (see Theorem \ref{analogthm42} below).

\begin{paragr}\label{defjoinordersonto}
Let $T$ be a tree endowed with an input edge (leaf) $e$.
\begin{equation}\label{defjoinordersonto1}
\begin{split}
\phantom{a}
\end{split}
T=
\begin{split}
\xymatrix@R=10pt@C=12pt{
&&&&*=0{}&\\
&*=0{}\ar@{--}[rrrr]&&&*=0{}\ar@{-}[u]^e&*=0{}\\
&&&*=0{\bullet}\ar@{}[u]|{\dots}
\ar@{-}[d]^{}\ar@{-}[urr]\ar@{-}[ull]&&\\
&&&*=0{}&&
}
\end{split}\end{equation}
Given an integer $n\geq 0$, we define the tree $n\star_e T$
as the tree obtained by joining the $n$-simplex
to the edge $e$ by a new vertex $v$.
\begin{equation}\label{defjoinordersont2}
\begin{split}
\phantom{a}
\end{split}
n\star_e T=
\begin{split}
\xymatrix@R=10pt@C=12pt{
&&&&*=0{}&\\
&&&&*=0{\bullet}\ar@{-}[u]^0&\\
&&&&*=0{}\ar@{-}[u]^1&\\
&&&&*=0{\bullet}\ar@{}[u]|(.75)\vdots&\\
&&&&*=0{\bullet}\ar@{-}[u]^n_(.1)v&\\
&*=0{}\ar@{--}[rrrr]&&&*=0{}\ar@{-}[u]^e&*=0{}\\
&&&*=0{\bullet}\ar@{}[u]|{\dots}
\ar@{-}[d]^{}\ar@{-}[urr]\ar@{-}[ull]&&\\
&&&*=0{}&&
}
\end{split}\end{equation}
This defines a unique functor
\begin{equation}\label{defjoinordersonto3}
\Delta\To\Omega \ , \quad
[n]\longmapsto n\star_e T
\end{equation}
such that the obvious inclusions $i[n]\To n\star_e T$
are functorial.
We thus get a functor
\begin{equation}\label{defjoinordersonto4}
(-)\star_e T:\Delta\To T/\dset
\end{equation}
(where $T/\dset$ denotes the category of
dendroidal sets under $\Omega[T]$).
By Kan extension, we obtain a colimit preserving functor
\begin{equation}\label{defjoinordersonto5}
(-)\star_e T:\sset\To T/\dset\, .
\end{equation}
We have $\Delta[n]\star_e T=\Omega[n\star_e T]$.
The functor \eqref{defjoinordersonto5} has a right
adjoint
\begin{equation}\label{defjoinordersonto6}
(-){/_e}T:T/\dset\To\sset\, .
\end{equation}
\end{paragr}

\begin{rem}[Functoriality in $T$]\label{functjoinoverinT}
We shall say that a face map $R\To T$ is \emph{$e$-admissible}
if it does not factor through the external face map which chops off $e$.
For such a face $R\To T$, $e$ is also a leaf of $R$, and there are natural maps
\begin{equation}\label{defjoinordersonto8}
n\star_e R\To n\star_e T\, .
\end{equation}
Thus, we obtain, for each simplicial set $K$, and each dendroidal
set $X$ under $T$ (i.e. under $\Omega[T]$), natural maps
\begin{equation}\label{defjoinordersonto9}
K\star_e R\To K\star_e T
\end{equation}
and
\begin{equation}\label{defjoinordersonto10}
X/_eT\To X/_eR\, .
\end{equation}
Similarly, the inclusions $\Omega[n]\To\Delta[n]\star_e T$
induce a projection
\begin{equation}\label{defjoinordersonto7}
X/_eT\To i^*(X)
\end{equation}
for any dendroidal set $X$ under $T$.
\end{rem}

\begin{paragr}\label{defkeadmiss}
Let $0<i\leq n$ be integers.
Let $\{R_1,\ldots,R_t\}$, $t\geq 1$, be a finite family of
$e$-admissible faces of $T$, and define
$$C\subset D\subset\Omega[n\join_e T]$$
by
$$C=\big(\bigcup^t_{s=1}\Lambda^i[n]\join_e R_s\big)\cup\Omega[n]
\quad\text{and}\quad
D=\bigcup^t_{s=1}\Delta[n]\join_e R_s\, ,$$
where $\Omega[n]$ is seen as a subcomplex of $\Omega[n\join_e T]$
through the canonical embedding.
\end{paragr}

\begin{lemma}\label{lemmakeadmiss}
Under the assumptions of \ref{defkeadmiss}, the map $C\To D$
is an inner anodyne extension.
\end{lemma}

\begin{proof}
For $p\geq 1$, write $\mathcal{F}_p$ for the set
of faces $F$ of $\Omega[n\join_e T]$ which
belong to $D$ but not to $C$, and which are of the form
$F=\Omega[n\join_e R]$ for an $e$-admissible face $R$ of $T$
with exactly $p$ edges.
Define a filtration
$$C=C_0\subset C_1\subset\ldots\subset C_p\subset\ldots\subset D$$
by
$$C_p=C_{p-1}\cup\bigcup_{F\in\mathcal{F}_p} F\ , \quad p\geq 1\, .$$
We have $D=C_p$ for $p$ big enough, and it is sufficient to prove
that the inclusions $C_{p-1}\To C_p$ are inner anodyne for $p\geq 1$.
If $F$ and $F'$ are in $\mathcal{F}_p$, then $F\cap F'$
is in $C_{p-1}$. Moreover, if $F=\Omega[n\join_e R]$
for an $e$-admissible face $R$ of $T$, then we have
$$F\cap C_{p-1}=\Lambda^i[n\join_e R]\, ,$$
which is an inner horn.
Hence we can describe the inclusion $C_{p-1}\To C_p$
as a finite composition of pushouts by inner horn inclusions
of shape $F\cap C_{p-1}\To F$ for $F\in\mathcal{F}_p$.
\end{proof}

\begin{prop}\label{joineinnerKani0}
Let $0<i\leq n$ be integers.
The inclusion
$$(\Lambda^i[n]\star_e T)\cup\Omega[n]\To\Omega[n\star_e T]$$
is an inner anodyne extension.
\end{prop}

\begin{proof}
Apply Lemma \ref{lemmakeadmiss}.
\end{proof}

\begin{prop}\label{joineinnerKani}
For any inner Kan fibration $p:X\To Y$ under $T$,
the morphism $X/_eT\To Y/_e\times_{i^*(Y)}i^*(X)$
is a right fibration of simplicial sets.

In particular, for any $\infty$-operad $X$ under $T$, the map
$X/_eT\To i^*(X)$ is a right fibration between $\infty$-categories.
\end{prop}

\begin{proof}
This follows from Proposition \ref{joineinnerKani0}
by a standard adjunction argument.
\end{proof}

\begin{thm}\label{analogthm42}
Let $S$ be a tree with at least two vertices, let $v$ be a unary top vertex in $S$,
and let $p:X\To Y$ be an inner Kan fibration between $\infty$-operads.
Then any solid commutative square of the form
$$\xymatrix{
\Lambda^v[S]\ar[r]^\varphi\ar[d]&X\ar[d]^p\\
\Omega[S]\ar[r]^\psi\ar@{..>}[ur]^h&Y
}$$
in which $\varphi(v)$ is weakly invertible in $X$ has a diagonal
filling $h$.
\end{thm}

\begin{proof}
The tree $S$ has to be of shape $S=1\star_e T$
for a tree $T$ with a given leaf $e$.
Under this identification, we have
$\Lambda^v[S]=\Lambda^0[1\star_eT]$.
A lifting $h$ in the solid commutative square
$$\xymatrix{
\Lambda^0[1\star_eT]\ar[r]^\varphi\ar[d]&X\ar[d]^p\\
\Omega[1\star_eT]\ar[r]^\psi\ar@{..>}[ur]^h&Y
}$$
is thus equivalent to a lifting $k$
in the diagram
$$\xymatrix{
\{0\}\ar[r]^{\tilde{\varphi}}\ar[d]&P\ar[d]\\
\Delta[1]\ar[r]^{\tilde{\psi}}\ar@{..>}[ur]^k&Q
}$$
in which $P=X/_eT$ and $Q=U\times_WV$, with
$$
U=\varprojlim_R X/_eR\ , \quad
V=Y/_eT\ , \quad W=\varprojlim_R Y/_eR\ ,
$$
where $R$ ranges over all the proper $e$-admissible
faces of $T$. As in the proof of Theorem \ref{dendjoyal},
it is now sufficient to prove the three following
properties:
\begin{itemize}
\item[(i)] the map $P\To Q$ is a right fibration;
\item[(ii)] $Q$ is an $\infty$-category.
\item[(iii)] if $\varphi(x)$ is weakly invertible in $X$, then
so is the $1$-cell $\tilde{\psi}$ in $Q$.
\end{itemize}
Properties (ii) and (iii) will follow from the two assertions below:
\begin{itemize}
\item[(iv)] the map $V\To W$ is a right fibration;
\item[(v)] the map $U\To i^*(X)$ is a right fibration.
\end{itemize}
As (iv) is a particular case of (i), we are thus reduced to prove (i) and (v).
\begin{sslem}
Proof of \emph{(i)}.
\end{sslem}
A lifting problem of shape
\begin{equation*}\begin{split}
\xymatrix{
\Lambda^i[n]\ar[r]^{}\ar[d]&P\ar[d]\\
\Delta[n]\ar[r]^{}\ar@{..>}[ur]&Q
}\end{split}\quad\qquad 0<i\leq n
\end{equation*}
is equivalent to a lifting problem of shape
\begin{equation*}\begin{split}
\xymatrix{
C\ar[r]^{}\ar[d]&X\ar[d]^p\\
\Omega[n\star_e T]\ar[r]^{}\ar@{..>}[ur]&Y
}\end{split}
\end{equation*}
where $C$ is the union of $\Lambda^i[n]\star_e T$ with
the union of the faces of $\Omega[n\star_e T]$
which are of the form $n\star_e S\To n\star_e T$,
where $S$ ranges over the $e$-admissible elementary faces of $T$.
In other words, $C=\Lambda^i[n\star_e T]$ is an inner horn, so that
the required lifting exists, because $p$ is assumed to be an inner Kan
fibration.
\begin{sslem}
Proof of \emph{(v)}.
\end{sslem}
A lifting problem of shape
\begin{equation*}\begin{split}
\xymatrix{
\Lambda^i[n]\ar[r]^{}\ar[d]&U\ar[d]\\
\Delta[n]\ar[r]^{}\ar@{..>}[ur]&i^*(X)
}\end{split}\quad\qquad 0<i\leq n
\end{equation*}
is equivalent to a lifting problem of shape
\begin{equation*}\begin{split}
\xymatrix{
C\ar[r]^{}\ar[d]&X\\
D\ar@{..>}[ur]&
}\end{split}
\end{equation*}
in which the inclusion $C\To D$ can be described as follows:
$$C=\Omega[n]\cup\bigcup_R \Lambda^i[n]\star_e R
\subset D=\bigcup_R \Omega[n\star_e R]
\subset\Omega[n\star_e T]\, ,$$
where $R$ ranges over the $e$-admissible elementary faces of $T$.
It is easily seen that the inclusion $C\To D$ is an inner anodyne extension by Lemma \ref{lemmakeadmiss}.
\end{proof}

\section{Another subdivision of cylinders}\label{appB}

\begin{paragr}
We will refer to the horn
inclusions of shape $\Lambda^x[S]\To\Omega[S]$,
where $S$ is a tree with a unary top vertex $x$,
as \emph{end extensions}.
A composition of pushouts of end extensions will be
called an \emph{end anodyne} map.

The goal of this section is to prove a dual version
of Theorem \ref{thmsubcyl}, namely:
\end{paragr}

\begin{thm}\label{thmsubcyldual}
Let $T$ be a tree with at least one vertex, and consider
the subobject
$$B_0=\{0\}\otimes\Omega[T]\cup
\Delta[1]\otimes\bord\Omega[T]\subset\Delta[1]\otimes\Omega[T]\, .$$
There exists a filtration of $\Delta[1]\otimes\Omega[T]$ of the 
form
$$B_0\subset B_1\subset\ldots\subset B_{N-1}\subset B_N=\Delta[1]\otimes\Omega[T]$$
where, for each $i$, $0\leq i<N$, the map $B_i\To B_{i+1}$
is either inner anodyne or end anodyne.

Moreover, the end anodyne maps are all push outs of the form
$$\xymatrix{
\Lambda^v[S]\ar[r]\ar[d]&B_i\ar[d]\\
\Omega[S]\ar[r]&B_{i+1}
}$$
with the following properties:
\begin{itemize}
\item[(i)] the tree $S$ has at least two vertices, and $v$ is
a unary top vertex;
\item[(ii)] the map
$$\Delta[1]\To\Lambda^v[S]\To B_i\subset\Delta[1]\otimes\Omega[T]\, ,$$
corresponding to the vertex $v$ in $S$, coincides with an inclusion
of shape
$$\Delta[1]\otimes\{t\}\To\Delta[1]\otimes\Omega[T]$$
for some edge $t$ in $T$.
\end{itemize}
\end{thm}

\begin{paragr}
As in the proof of Theorem \ref{thmsubcyl}, we will follow the convention
of \cite{dend2}, and write
$$\Omega[S]\otimes\Omega[T]=\bigcup^m_{i=1}\Omega[T_i]\, ,$$
where the union ranges over the partially ordered set of percolation
schemes, starting with a  number of copies of $T$ grafted on top of $S$,
and ending with the reverse grafting. For
\begin{equation*}
\begin{split}
\phantom{a}
\end{split}
\text{$S=[1]=$}\quad
\begin{split}
\xymatrix@R=10pt@C=12pt{
*=0{}\\
*=0{\circ}\ar@{-}[u]\\
*=0{}\ar@{-}[u]}
\end{split}\quad\text{,}
\end{equation*}
the first tree is of shape
\begin{equation}
\begin{split}
\phantom{a}
\end{split}T_1=\quad
\begin{split}
\xymatrix@R=10pt@C=12pt{
*=0{}&&*=0{}&&*=0{}&\\
&&*=0{\bullet}\ar@{}[u]|{T}\ar@{-}[ull]^{}\ar@{-}[urr]_{}&&\\
&&*=0{\circ}\ar@{-}[u]&&\\
&&*=0{}\ar@{-}[u]&&
}
\end{split}\quad\text{,}\end{equation}
and the last one is
\begin{equation}
\begin{split}
\phantom{a}
\end{split}T_m=\quad
\begin{split}
\xymatrix@R=10pt@C=12pt{
*=0{}&*=0{}&*=0{}&*=0{}&*=0{}&\\
*=0{\circ}\ar@{-}[u]&*=0{\circ}\ar@{-}[u]
\ar@{}[rr]|{\dots}
&*=0{}&*=0{\circ}\ar@{-}[u]&*=0{\circ}\ar@{-}[u]&\\
*=0{}\ar@{-}[u]&*=0{}\ar@{-}[u]&*=0{}&*=0{}\ar@{-}[u]&*=0{}\ar@{-}[u]&\\
&&*=0{\bullet}\ar@{}[u]|{T}
\ar@{-}[ull]^{}
\ar@{..}[ul]^{}
\ar@{-}[urr]_{}
\ar@{..}[ur]^{}&&\\
&&*=0{}\ar@{-}[u]&&
}
\end{split}\quad\text{.}
\end{equation}

Let us fix a linear order on the percolation schemes for
$\Delta[1]\otimes\Omega[T]$ which extends the natural partial order.
Such a linear ordering induces a filtration on the tensor product
$\Delta[1]\otimes\Omega[T]$,
\begin{equation}
C_0\subset C_1\subset\ldots\subset C_{m-1}\subset C_m=\Delta[1]\otimes\Omega[T]
\end{equation}
by setting
\begin{equation}
C_0=B_0=\{0\}\otimes\Omega[T]\cup
\Delta[1]\otimes\bord\Omega[T]\ \text{and} \
C_i=B_0\cup\Omega[T_1]\cup\dots\cup\Omega[T_i]\, .
\end{equation}
the filtration of Theorem \ref{thmsubcyldual} will be a refinement of
this one.

Let us start by considering $T_1$. If the root edge of $T$ is called $r$,
then $T_1$ looks like
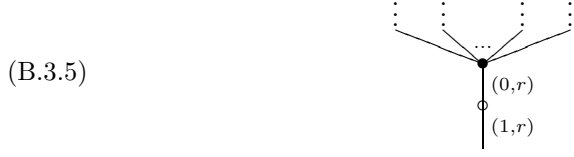
\begin{equation}\begin{split}\xymatrix@R=10pt@C=12pt{
*=0{}&*=0{}&*=0{}&*=0{}&*=0{}\\
*=0{}\ar@{}[u]|(.9)\vdots
&*=0{}\ar@{}[u]|(.9)\vdots
&*=0{}
&*=0{}\ar@{}[u]|(.9)\vdots
&*=0{}\ar@{}[u]|(.9)\vdots\\
&&*=0{\bullet}\ar@{}[u]|\dots\ar@{-}[ull]^{}\ar@{-}[urr]_{}
\ar@{-}[ul]^{}\ar@{-}[ur]_{}&&\\
&&*=0{\circ}\ar@{-}[u]_{(0,r)}&&\\
&&*=0{}\ar@{-}[u]_{(1,r)}&&
}\end{split}\end{equation}
With the exception of the faces $\partial_{(0,r)}(T_1)$ (which
contracts $(0,r)$) and $\partial_{(1,r)}(T_1)$ (which chops off
$(1,r)$ as well as the white vertex), any face $F$ of $T_1$ misses a colour of
$T$ (by this, we mean there is an edge $a$ in $T$ such that no edge in $F$
is named $(i,a)$). Hence, $\Omega[F]\subset\Delta[1]\otimes\bord\Omega[T]$
for these $F$. Moreover, $\partial_{(1,r)}(T_1)=\{0\}\otimes T_1$.
So $\Omega[T_1]\cap B_0=\Lambda^{(0,r)}[T_1]$, and
\begin{equation}\begin{split}\xymatrix{
\Lambda^{(0,r)}[T_1]\ar[r]\ar[d]&B_0\ar[d]\\
\Omega[T_1]\ar[r]&B_0\cup\Omega[T_1]
}\end{split}\end{equation}
is a pushout. So, if we let $B_1=C_1$, then $B_0\To B_1$
is obviously inner anodyne.

Suppose we have defined a filtration up to some $B_l$
\begin{equation}
B_0\subset B_1\subset\ldots\subset B_{l}\qquad l\geq 1\, ,
\end{equation}
so that $B_l=C_k$ for some $k$, $1\leq k\leq m$.
We will extend this filtration as $B_l\subset B_{l+1}\subset\ldots\subset B_{l'}$,
so that $B_{l'}=C_{k+1}$. The percolation scheme $T_{k+1}$ is obtained from
an earlier one $T_j$ by pushing a white vertex in $T_j$ one step up
through a black vertex $x$, as in
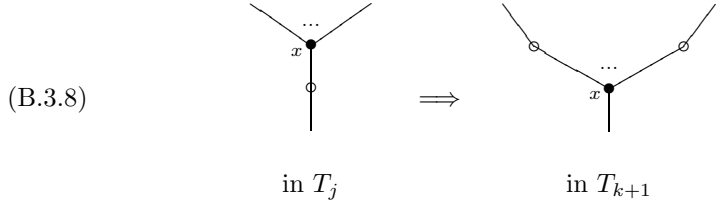
\begin{equation}
\begin{split}
\xymatrix@R=10pt@C=6pt{
&*=0{}&*=0{}&*=0{}&\\
&&*=0{\bullet}\ar@{}[u]|{\dots}\ar@{-}[ul]^{}\ar@{-}[ur]_{}&\\
&&*=0{\circ}\ar@{-}[u]^(.85){x}&\\
&&*=0{}\ar@{-}[u]&\\
&&\text{in $T_{j}$}&
}\end{split}
\Longrightarrow\quad
\begin{split}\xymatrix@R=10pt@C=6pt{
*=0{}&*=0{}&*=0{}&*=0{}&*=0{}&\\
&*=0{\circ}\ar@{-}[ul]&*=0{}&*=0{\circ}\ar@{-}[ur]&\\
&&*=0{\bullet}\ar@{-}[ul]\ar@{-}[ur]\ar@{}[u]|\dots &&\\
&&*=0{}\ar@{-}[u]^(.85)x&&\\
&&\text{in $T_{k+1}$}&&
}\end{split}
\end{equation}
(we have denoted by $x$ the black vertex in both trees,
although it would be more accurate to write $x$ for the relevant vertex of $T$,
and write $0\otimes x$ and $1\otimes x$ for the corresponding vertices in $T_j$
and $T_{k+1}$ repectively). The Boardman-Vogt relation states that, as subobjects
of $\Delta[1]\otimes\Omega[T]$, the face of $T_{k+1}$ obtained by contracting
all input edges of $x$ coincides with the face of $T_j$ obtained
by contracting the output edge of $x$ in $T_j$. In particular,
notice that if $x$ has no input edges at all (i.e. if $x$ is a `nullary
operation' in $T$), then $T_{k+1}$ is a face of $T_j$, so $C_{k+1}=C_k$,
and we let $B_{l'}=B_l$, and there is nothing prove.
Therefore, from now on, we will assume that the set of input edges
of $x$, denoted $\mathit{input}(x)$, is non-empty, and we proceed
as follows.

Let $E$ be the set of all colours (edges) $e$ in $T$ for which
\begin{equation}
\begin{split}
\xymatrix@R=10pt@C=12pt{
*=0{}\\
*=0{\circ}\ar@{-}[u]_{(0,e)}\\
*=0{}\ar@{-}[u]_{(1,e)}}
\end{split}
\end{equation}
occurs in $T_{k+1}$. For $U\subset E$, let
\begin{equation}
T^{(U)}_{k+1}\subset T_{k+1}
\end{equation}
be the face given by contracting all the edges $(1,e)$
for $e\in E$ but $e$ \emph{not} in $U$. Notice that if $U\cap \mathit{input}(x)=\varnothing$,
then $\Omega[T^{(U)}_{k+1}]\subset B_l$ by the Boardman-Vogt relation
just mentioned. Therefore, we will only consider $U$ with
$U\cap \mathit{input}(x)\neq\varnothing$.
We will successively adjoin $\Omega[T^{(U)}_{k+1}]$ to $B_l$
for larger and larger such $U$, until we reach the case where $U=E$
and $T^{(U)}_{k+1}=T_{k+1}$.

If $U=\{e\}$ is a singleton (with $e$ an input edge of $x$),
then the face $\partial_{(1,e)}\Omega[T^{(\{e\})}_{k+1}]$
is contained in $B_l$ as said, while the face
$\partial_{(0,e)}\Omega[T^{(\{e\})}_{k+1}]$ is not
(it cannot belong to an earlier $T_i$,
and is obviously not contained in $B_0=C_0$).
Any other face of $T^{(\{e\})}_{k+1}$ misses a colour
of $T$ and hence is contained in $B_0$. Thus,
\begin{equation}
\Omega[T^{(\{e\})}_{k+1}]\cap B_l\subset \Omega[T^{(\{e\})}_{k+1}]
\end{equation}
is either an inner horn (if $(0,e)$ is an inner edge
of $T_{k+1}$) or an end extension (if $(1,e)$ is an input edge of $T_{k+1}$).
In either case, we can adjoin $\Omega[T^{(\{e\})}_{k+1}]$
by forming the pushout below.
\begin{equation}\begin{split}\xymatrix{
\Omega[T^{(\{e\})}_{k+1}]\cap B_l\ar[r]\ar[d]&B_l\ar[d]\\
\Omega[T^{(\{e\})}_{k+1}]\ar[r]&\Omega[T^{(\{e\})}_{k+1}]\cup B_l
}\end{split}\end{equation}
We successively adjoin $\Omega[T^{(\{e\})}_{k+1}]$ to $B_l$
in this way for all $e$ in $E$ which are input edges of $x$ in $T$:
if these are $e_1,\ldots,e_p$, let
\begin{equation}
B_{l+r}=B_l\cup\Omega[T^{(\{e_1\})}_{k+1}]\cup\dots\cup\Omega[T^{(\{e_r\})}_{k+1}]\, .
\end{equation}
Then, for each $r<p$, the map $B_{l+1}\To B_{l+r+1}$ is inner anodyne
or end anodyne.

In general, we proceed by induction on $U$.
Choose $U\subset E$ with $U\cap \mathit{input}(x)\neq\varnothing$, and
assume we have adjoined $\Omega[T^{(U')}_{k+1}]$ already, for all $U'$
of smaller cardinality than $U$. We will write $B_{l''}$ for the last object in
the filtration constructed up to that point.
Fix an order on the set of elements of $U$, and write it as
\begin{equation}
U=\{\alpha_1,\ldots,\alpha_s\}\, .
\end{equation}
Consider $\Omega[T^{(U)}_{k+1}]$. The tree $T^{(U)}_{k+1}$
has edges $(0,c)$ or $(1,c)$ for $c$ not in $E$, and the corresponding
face misses the colour $c$ alltogether, hence
$\partial_{(i,c)}\Omega[T^{(U)}_{k+1}]$ is contained in $B_0$ for these $c$.
Next, the tree $T^{(U)}_{k+1}$ has edges coloured $(1,\alpha_i)$ for
$1\leq i\leq s$, and contracting any of these gives a face
\begin{equation}
\partial_{(1,\alpha_i)}\Omega[T^{(U)}_{k+1}]=\Omega[T^{(U-\{\alpha_i\})}_{k+1}]
\end{equation}
which is contained in $B_{l''}$ by the inductive assumption on $U$.
None of the faces given by contracting (if it is inner) or by
chopping off (if it is outer) an edge $(0,\alpha_i)$ in
$T^{(U)}_{k+1}$ can be contained in $B_{l''}$, however.

Let $A_1,\ldots,A_t$ be all the subsets of the set of these edges
$\{(0,\alpha_1),\ldots,(0,\alpha_s)\}$ of $T^{(U)}_{k+1}$
which contain $(0,\alpha_1)$, and order them by some linear order
extending the inclusion order. So there are $t=2^{s-1}$ such $A_i$, and
we could fix the order to be
$$\begin{aligned}
A_1&=\{(0,\alpha_1)\}\\
A_2&=\{(0,\alpha_1),(0,\alpha_2)\}\\
\vdots&\phantom{=}\qquad\vdots\\
A_s&=\{(0,\alpha_1),(0,\alpha_s)\}\\
A_{s+1}&=\{(0,\alpha_1),(0,\alpha_2),(0,\alpha_3)\}\\
\vdots&\phantom{=}\qquad\vdots\\
A_t&=\{(0,\alpha_1),\ldots,(0,\alpha_s)\}\, .
\end{aligned}$$
For $q=1,\ldots,t$, let $T^{(U,q)}_{k+1}$ be the
tree obtained from $T^{(U)}_{k+1}$ by contracting or chopping off
all the edges $(0,\alpha_i)$ not in $A_q$.
So
\begin{equation}
T^{(U,1)}_{k+1}=\partial_{(0,\alpha_{s})}\partial_{(0,\alpha_{s-1})}\dots
\partial_{(0,\alpha_{2})}T^{(U)}_{k+1}\, ,
\end{equation}
and
\begin{equation}
T^{(U,t)}_{k+1}=T^{(U)}_{k+1}\, .
\end{equation}
We will successively adjoin these $\Omega[T^{(U,q)}_{k+1}]$ to the filtration, to form
the part
\begin{equation}
B_{l''}\subset B_{l''+1}\subset\ldots\subset B_{l''+t}=B_{l''}\cup\Omega[T^{(U)}_{k+1}]
\end{equation}
of the filtration, as
\begin{equation}
B_{l''+q}=B_{l''}\cup\Omega[T^{(U,1)}_{k+1}]\cup\dots\Omega[T^{(U,q)}_{k+1}]\, .
\end{equation}
We start with $T^{(U,1)}_{k+1}$. The only face of $\Omega[T^{(U,1)}_{k+1}]$
not contained in $B_{l''}$ is the one given by the edge $(0,\alpha_1)$. Thus
\begin{equation}
\Omega[T^{(U,1)}_{k+1}]\cap B_{l''}\subset \Omega[T^{(U,1)}_{k+1}]
\end{equation}
is either an inner horn (if $(0,\alpha_1)$ is an inner edge) or an end extension
(if $(0,\alpha_1)$ is an input edge of $T^{(U)}_{k+1}$, i.e. $\alpha_1$
is an input edge of $T$). So the pushout $B_{l''}\To B_{l''+1}$
is either inner anodyne or end anodyne.

Suppose we have adjoined $\Omega[T^{(U,q')}_{k+1}]$ for all $1\leq q'<q$,
so have arrived at the stage $B_{l''+q-1}$ of the filtration.
Consider now $A_q$ and the corresponding dendroidal set $\Omega[T^{(U,q)}_{k+1}]$.
As before, its faces given by edges coloured by $(i,c)$ for $i=0,1$
with $c$ not in $E$ are contained in $B_0$, and its faces given by
edges coloured $(1,e)$ with $e\in U$ are contained in
$\Omega[T^{(U')}_{k+1}]$ for a smaller $U'=U-\{e\}$, hence are contained
in $B_{l''}$. Let us consider the remaining faces given by
the edges $(0,\alpha_1),\ldots,(0,\alpha_r)$ in $A_q$.
If $i\neq 1$, the face of $\Omega[T^{(U,q)}_{k+1}]$
given by $(0,\alpha_i)\in A_q$ is contained in
$\Omega[T^{(U,q')}_{k+1}]$ for some $q'<q$ with
$A_{q'}=A_q-\{(0,\alpha_i)\}$. So the only face that is missing is
the one given by $(0,\alpha_1)$, i.e.
\begin{equation}
\Omega[T^{(U,q)}_{k+1}]\cap B_{l''+q-1}=\Lambda^{(0,\alpha_1)}[T^{(U,q)}_{k+1}]\, .
\end{equation}
Therefore, the induced pushout $B_{l''+q-1}\To B_{l''+q}$
is either inner anodyne or end anodyne (depending on whether
$\alpha_1$ is an external edge of $T$ or not).
At the end, when $q=t$, we have adjoined all of $\Omega[T^{(U)}_{k+1}]$.

This completes the contruction of the segment of the filtration
for the subset $U\subset E$. As said, we continue this construction
until we reach the stage $U=E$, when $T^{(U)}_{k+1}=T_{k+1}$, which
completes the construction of the segment of the filtration from
$B_{l}$ until $B_{l'}$, interpolating between $C_k$ and $C_{k+1}$.
This completes the description of the filtration. From it,
the last part of the theorem is clear.
\end{paragr}

\end{appendices}

\section*{Erratum}


Proposition 1.9
is wrong as stated, and should be modified as we will explain below. This modification does not affect any of the main results of this paper and its two sequels \cite{dend4, dend5}: the existence of the model structure on dendroidal sets of the present paper,
the equivalent model structures for dendroidal complete Segal spaces and for Segal operads in \cite{dend4}, and the Quillen equivalence to the model category of simplicial  operads in \cite{dend5}. However, the error does affect all the statements concerning the monoidality of the model structures. 

Recall the inclusion $i: \Delta \rightarrow \Omega$, and the induced left Quillen functor
$i_!: \sset \rightarrow \dset$. Let us call a tree (an object of $\Omega$) \emph{linear} if it lies in the image of $i$, and a dendroidal set \emph{simplicial} if it lies in the image of
$i_!$. Also, let us call a tree \emph{open} if it has no vertices of valence zero (i.e. no vertices without input edges). If $S \rightarrow T$ is a morphism in $\Omega$ and $T$ is open, then $S$ is necessarily open as well. Therefore, the open trees define a subobject $\open$ of the terminal object in $\dset$. Let us call a dendroidal set \emph{open} if the unique map to the terminal object factors through $\open$. These open dendroidal sets form a full subcategory $\dset / \open$ of $\dset$. More generally, we call a normal monomorphism $X \rightarrow Y$ \emph{linear} if it is obtained by attaching linear trees (i.e. $X \rightarrow Y$ lies in the saturation of $\partial\Omega[T] \rightarrow \Omega[T]$ for $T$ linear) and \emph{open} if it is obtained by attaching open trees. 

The modified version of Proposition 1.9 should be:

\setcounter{section}{1}
\setcounter{thm}{8}
\begin{prop} \label{prop:prop1}
Let $A \rightarrow  B$ and $X \rightarrow Y$ be normal monomorphisms.  If one of them is linear or both are open, then the induced map  
	\begin{equation*}
A\otimes Y \cup_{A \otimes X} B\otimes  X \rightarrow B\otimes Y
	\end{equation*}
is again a normal monomorphism (and is open as well in the second case).
\end{prop}

By the usual induction, this follows from the following lemma
(whose proof is an elementary but tedious combinatorial argument,
and we refer the reader to \cite{CMn} for the details), which should
be added at the very end of the first section:

\setcounter{thm}{10}

\begin{lemma} Let $S$ and $T$ be trees. If one of them is linear or both are open, then the pushout-product map 
\begin{equation*}
\partial\Omega[S] \otimes \Omega[T]
\cup_{\partial\Omega[S] \otimes \partial\Omega[T]} \Omega[S] \otimes \partial\Omega[T]
\rightarrow \Omega[S] \otimes \Omega[T]
\end{equation*}
is a normal monomorphism (and is open as well in the second case).
\end{lemma}

The main result of \cite{dend2} (namely Proposition 9.2) is wrong as stated,
but its proof says the following: given two trees $S$ and $T$, as well as an inner edge $e$
of $S$, if $W$ denotes the the image of the map
$$\Lambda^e[S] \otimes \Omega[T]
\cup_{\Lambda^e[S] \otimes \partial\Omega[T]} \Omega[S] \otimes \partial\Omega[T]
\rightarrow \Omega[S] \otimes \Omega[T]\, ,$$
then the inclusion
$$W\subset \Omega[S] \otimes \Omega[T]$$
is inner anodyne. This means that Proposition 3.1
should be replaced by the following statement.

\setcounter{section}{3}
\setcounter{thm}{0}

\begin{prop} \label{prop:prop2}
Let $A \rightarrow  B$ and $X \rightarrow Y$ be normal monomorphisms.  If one of them is linear or both are open, and if one of them is inner anodyne,
then the induced map  
	\begin{equation*}
A\otimes Y \cup_{A \otimes X} B\otimes  X \rightarrow B\otimes Y
	\end{equation*}
is again an inner anodyne extension.
\end{prop}

Proposition 3.3 should be modified accordingly
(replacing, in Proposition \ref{prop:prop2} above, the
expression ``inner anodyne'' by ``$J$-anodyne'').

Let us call a model category $\mathcal{M}$ \emph{Joyal simplicial} if it satisfies the axioms for a simplicial model category, but with respect to the Joyal model structure on simplical sets instead of the classical Kan-Quillen structure. The model structure on $\dset$ established in
this paper is not monoidal. Instead, Proposition 3.17
should be replaced by the following result,
which follows immediately from the propositions above and from the arguments
explained in the original `proof' of Prop. 3.17:

\setcounter{thm}{16}

\begin{prop} The Boardman-Vogt tensor product turns the
model structure on $\dset$ into a Joyal simplicial model structure,
and induces a symmetric monoidal model structure on the category
$\dset / \open$ of open dendroidal sets. 
\end{prop}

In the proof of the existence of the model structure, it is only the ``linear half'' of  Propositions \ref{prop:prop1} and \ref{prop:prop2}
which are used, and this proof is unaffected. Indeed,
Propositions 3.1 and 3.3
are used only in the case where one of the maps is in the image of $i_!$.

\bigskip
 
We close this erratum with a list of places where the reference to
(consequences of) monoidality should be reformulated in accordance with the previous two propositions: In the present paper, these are point 1 in the introduction, Proposition 2.6(c), Corollary 2.8(b), Corollary 2.9(b); In the proof of Lemma 6.15, the first sentence
should begin as `We know that the image by the functor $i^*$ of the map$\ldots$';
Also the statement between brackets at the very end
of Proposition 6.20 should be skipped (and the proof of Prop.~6.20
should refer to Boardman and Vogt's
explicit description of the category associated to a quasi-category, instead
of its dendroidal analogue).
In \cite{dend4,dend5},  no use of the alleged monoidality of the model structure is made,
and the necessary changes all concern inessential references to the monoidality:
in paper \cite{dend4}, the places where we recall the monoidality from the present paper,
are in the abstract and the introduction, in Theorem 1.1, Remarks 6.14 and  8.16,
and finally in the proof of Propposition 6.11, where
it is Proposition 3.1 above which should be used rather than the monoidality;
and in paper \cite{dend5}, the mentioning of monoidality is in the introduction as well as in Proposition 2.8 and Theorem 5.7.

We repeat that this error only affects the monoidality of the model structures, and none of the main results about the existence of the model structure on $\dset$ and the Quillen equivalent model categories presented in \cite{dend4,dend5}.

\nocite{dend1,dend2,Ci3,joyal,joytier4,lurie,GZ}
\bibliography{bibdend}
\bibliographystyle{amsalpha}
\end{document}